%% file: SR_pretzel.tex
\documentclass[11pt,reqno]{amsart}

\usepackage{allneeded}
\usepackage{enumerate}
\usepackage{a4wide}
\usepackage{xypic}
\usepackage[title,titletoc,toc]{appendix}

\begin{document}

\title{On the Slice-Ribbon Conjecture for pretzel knots}%

\author{Ana G. Lecuona}%

\email{ana.lecuona@latp.univ-mrs.fr}%

\address{LATP, Aix-Marseille Universit\'e, Marseille, FRANCE}%

\thanks{This work is partially supported by the Spanish GEOR MTM2011-22435.}


\date{September 2, 2013}%

\begin{abstract}
We give a necessary, and in some cases sufficient, condition for sliceness inside the family of pretzel knots $P(p_1,...,p_n)$ with one $p_i$ even. The three stranded case yields two interesting families of examples: the first consists of knots for which the non-sliceness is detected by the Alexander polynomial while several modern obstructions to sliceness vanish. The second family has the property that the correction terms from Heegaard-Floer homology of the double branched covers of these knots do not obstruct the existence of a rational homology ball; however, the Casson-Gordon invariants show that the double branched covers do not bound rational homology balls.
\end{abstract}
\maketitle

\begin{quote}
\begin{footnotesize}
\tableofcontents
\end{footnotesize}
\end{quote}

\section{Introduction}\label{s:int}

Pretzel knots have been thoroughly studied since they were first introduced by Reidemeister in \cite{b:Re}. In recent work Greene and Jabuka have determined the order in the smooth knot concordance group of $3$--stranded pretzel knots $P(p_1,p_2,p_3)$ with all $p_i$ odd \cite[Theorem~1.1]{b:GJ}. A corollary of their result is that the slice--ribbon conjecture proposed by Fox in \cite{b:Fo} holds true for this family of knots. In this paper we address the question of sliceness in the family of pretzel knots $P(p_1,...,p_n)$ with one $p_i$ even. Our main result, Theorem~\ref{t:slice}, gives a necessary condition for sliceness in this family. The condition obtained is not sufficient, but we propose a conjecture of what constraints should be added to Theorem~\ref{t:slice} to obtain a complete characterisation of ribbon pretzel knots. 

As a byproduct of our research we have found two interesting families of pretzel knots. The first one is a one parameter family of knots for which most of the known slice obstructions vanish:  the signature, the determinant, the Arf invariant, the $s$-invariant and the $\tau$ invariant among others. However, the Alexander polynomial is able to show that more than three quarters of the knots in this family are not slice. This property makes this set of knots an excellent source to test future slice invariants. The second interesting family consists of a set of pretzel knots whose double branched covers do not bound rational homology balls. The non-existence of these balls is determined via the Casson-Gordon invariants which, for this particular family of $3$-manifolds, turn out to be a more powerful tool than the $d$-invariants from Heegaard Floer homology in obstructing the existence of a rational ball.

Given nonzero integers $p_1,...,p_n$ the pretzel link $P(p_1,...,p_n)$ is obtained by taking $n$ pairs of parallel strands, introducing $p_i$ half twists on the $i$-th pair, with the convention $p_{i}>0$ for right--hand crossings and $p_{i}<0$ for left--hand crossings, and connecting the strands with $n$ pairs of bridges. As an example, the first knot in Figure~\ref{f:cob_pretzel} corresponds to $P(-5,5,-3,3,7)$. If more than one of the $p_i$ is even or if $n$ is even and none of the $p_i$ is even then $P(p_{1},...,p_{n})$ is a link. In all other cases it is a knot. Inside the family of pretzel \emph{knots} $P(p_1,...,p_n)$ we limit our considerations to those with one even parameter and moreover from now on we fix $n\geq 3$ and $|p_i|>1$ for all $i$. Note that if $n\leq 2$ or if $n=3$ and one of the $p_i$ satisfies $p_i=\pm 1$, then the pretzel knot is a $2$--bridge, already studied in \cite{b:Li,b:Li2}.

This paper addresses the question of sliceness of pretzel knots. A knot in $S^{3}$ is said to be \emph{slice} if it bounds a disc smoothly embedded in the $4$-ball. The double cover of $S^{3}$ branched over a slice knot bounds a rational homology 4-ball. This property is one of the main obstructions we shall study. The strategy we follow is straightforward: use different obstructions to rule out the non-slice pretzel knots and show that the remaining knots are slice by explicitly constructing the slice discs in $B^{4}$. The discs we find have the property of being \emph{ribbon}, i.e.\ they have no local maxima for the radial function in $B^{4}$. The slice--ribbon conjecture says that all slice knots are ribbon and the results in the present work support this conjecture. 

It is well known \cite[Theorem~12.19]{b:BZ} (recall that pretzel knots are a particular case of the more general family of Montesinos links) that among the $n!$ permutations of the parameters $(p_1,...,p_n)$, the $2n$ of them which correspond to cyclic permutations, order reversing permutations and compositions of these leave invariant the knot $P(p_1,...p_n)$. Two pretzel knots $K$ and $K^*$ which are not isotopic but who share the same set of parameters are \emph{mutants}, that is $K^*$ can be obtained from $K$ by removing a $3$-ball from $S^3$ that meets $K$ in two proper arcs and gluing it back via an involution $\tau$ of its boundary $S$, where $\tau$ is orientation preserving and leaves the set $S\cap K$ invariant. All pretzel knots defined by the same set of parameters $p_1,...,p_n$ have the same double branched cover. 

Our main result, Theorem~\ref{t:slice}, is stated for pretzel knots up to reordering of the parameters because the obstructions to sliceness that we analyze live in the double branched cover of these knots. Up to reordering we are able to establish the sliceness of pretzel knots with one even $p_{i}$ except for the following set
$$\mathcal E=\{a,-a-2,-\frac{(a+1)^{2}}{2},q_{1},-q_{1},\dots,q_{m},-q_{m}\},$$
where $m\geq 0$, $a,|q_{i}|\geq 3$ odd and $a\equiv 1,11,37,47,49,59$ (mod 60). Our main result is the following.

\begin{thm}\label{t:slice}
Let $K=P(p_1,...,p_n)$ be a slice pretzel knot with one even parameter and such that $\{p_{1},\dots,p_{n}\}\not\subset\mathcal E$. Then, the $n$--tuple of integers $(p_1,...,p_n)$ can be reordered so that it has the form
\begin{itemize}
\item[$(1)$] $(q_1,-q_1\pm 1,q_2,-q_2,...,q_{\frac{n}{2}},-q_{\frac{n}{2}})$ if $n$ is even;
\item[$(2)$] $(q_0,q_1,-q_1,...,q_{\frac{n-1}{2}},-q_{\frac{n-1}{2}})$ if $n$ is odd.
\end{itemize}
\end{thm}

\begin{rem}
Ligang Long \cite{b:Long} has independently obtained Theorem~\ref{t:slice} for the case $n=4$. Moreover, he has several partial results concerning the sliceness of pretzel knots without the restriction of having one even parameter.  
\end{rem}

As further explained in Section~\ref{s:cor} not all possible orders of the parameters in Theorem~\ref{t:slice} yield slice knots. In Corollary~\ref{c:affirm} we show that for certain orders of the parameters the knots in the above Theorem are actually slice. Moreover we conjecture that the orders proposed in Conjecture~\ref{c:conj} are all the possible orders of the parameters in Theorem~\ref{t:slice}  yielding slice knots.

As detailed in Section~\ref{s:weird} most of the pretzel knots in the family $P(a,-a-2,-\frac{(a+1)^{2}}{2})$, $a\geq 3$ odd, are not slice. However establishing that none of them is slice is still an open challenge. There is a great amount of evidence supporting the following conjecture

\begin{conj}
If $\{p_{1},\dots,p_{n}\}\subset\mathcal E$ then the pretzel knot $P(p_{1},\dots,p_{n})$ is not slice.
\end{conj}

Note that a pretzel knot of the form $P(p_1,p_2,p_3)$ is independent of the order of the parameters, since cyclic permutations, order reversing permutations and compositions of these comprise all possible permutations of three elements. For $3$-stranded pretzel knots whose defining parameters are not in $\mathcal E$ an
easy corollary of Theorem~\ref{t:slice} is the validity of the slice--ribbon conjecture.

In the following Corollary~\ref{c:cor} the results for $P(p_1,p_2,p_3)$ with three odd parameters were already proved in \cite{b:GJ}. Our work proves the statement for $3$--stranded pretzel knots with one even parameter and leaves out the case $P(a,-a-2,-\frac{(a+1)^{2}}{2})$, $a\geq 3$ odd and $a\equiv 1,11,37,47,49,59$ (mod 60).

\begin{cor}\label{c:cor}
The slice--ribbon conjecture holds true for pretzel knots of the form $P(p_1,p_2,p_3)$ where $p_1,p_2,p_3\in\Z$ and $\{p_{1},p_{2},p_{3}\}\not\subset\mathcal E$.
\end{cor}

In order to prove Theorem~\ref{t:slice} we start using the approach of \cite{b:Li}, which is also followed in \cite{b:GJ}: if a pretzel knot $K$ is slice then its double branched cover $Y$ is the boundary of a rational homology ball $W$ \cite[Lemma~17.2]{b:Ka}. Moreover, up to considering the mirror image of $K$, the $3$--manifold $Y$ is also the boundary of a negative definite $4$--manifold $M$ obtained by plumbing together disc bundles over spheres. We can build a closed, oriented, negative definite, $4$--manifold $X$ as $M_{\cup_Y} (-W)$. By Donaldson's celebrated theorem \cite{b:Do} the intersection form $Q_X$ of $X$ must be diagonalizable and therefore, since $H_2(M;\Z)/\Tors\cong H_2(X;\Z)/\Tors\cong\Z^n$, there must exist a monomorphism $\iota:\Z^n\rightarrow\Z^n$ such that $Q_M(\alpha,\beta)=-\mathrm{Id}(\iota(\alpha),\iota(\beta))$ for every $\alpha,\beta\in\Z^n\cong H_2(M;\Z)/\Tors$. 

The existence of $\iota$ is enough to guarantee sliceness among $2$--bridge knots \cite{b:Li}. In the case of pretzel knots $P(p_1,p_2,p_3)$ with all $p_i$ odd, this obstruction shows that there must exist some $\lambda\in\Z$ such that $-p_3=p_1\lambda^2+p_2(\lambda+1)^2$ \cite[Proposition~3.1]{b:GJ}. However, not all these pretzel knots are slice. In fact, using the Ozsv\'ath--Szab\'o correction terms for rational homology spheres Greene and Jabuka conclude that only $\lambda=0,-1$ result in slice knots.

In our case, for pretzel knots $P(p_1,...,p_n)$ with one even parameter, even in the case of three strands, the existence of $\iota$ is a weaker obstruction to sliceness than in the cases studied in \cite{b:GJ,b:Li}. For example,
contrary to what happens for $2$--bridge knots and pretzel knots of the form $P(p_1,p_2,p_3)$ with all $p_i$ odd, in our case the existence of $\iota$ does not imply that the knot signature is zero. The proof of Theorem~\ref{t:slice} has three main steps: first we determine the pretzel knots with vanishing signature such that $\iota$ exists. Not all the resulting knots are slice. In a second step we use the correction terms from Heegaard-Floer homology to further restrain the family of candidates to slice knots. This leaves us with two one parameters families to further study. The sliceness of one of these families is ruled out using Casson-Gordon invariants while the other one is partially treated studying Alexander polynomials.\\

The rest of the paper is organized as follows. The proof of Corollary~\ref{c:cor} is carried out in Section~\ref{s:cor} assuming Theorem~\ref{t:slice}. An easy algorithm is given to detect ribbon pretzel knots and we show that many of the knots from Theorem~\ref{t:slice} are actually slice. In Section~\ref{s:pre} we recall some properties of the Seifert spaces associated to pretzel knots and of the knot signature. Section~\ref{s:two} deals with the two interesting families of pretzel knots whose properties were described above. Finally, Sections~\ref{s:1step} and \ref{s:step2} treat the general case combining Donaldson's theorem with the knot signature and the correction terms from Heegaard Floer homology.
\\

\begin{quote}
\footnotesize
\textbf{Acknowledgments.} This project began while I was a Ph.D.\ student under the supervision of Paolo Lisca to whom I am very grateful and from whom I still learn so much. I have had the opportunity to discuss this problem with many people who have all contributed to my better understanding of its different aspects: Benjamin Audoux, David Cimasoni, Jose F. Galvan, Jose M. Gamboa, Farshid Hajir, Lukas Lewark, Marco Mazzucchelli, Brendan Owens and Stepan Orevkov. Very special thanks go to Josh Greene and Slaven Jabuka, not only for their inspiring paper on this subject that introduced me to pretzel knots, but especially for their great generosity and the  many hours we have shared thinking about this problem together. 
\end{quote}

\section{Slice pretzel knots}\label{s:cor}

In this section we prove that, for certain orders of the parameters, the knots in Theorem~\ref{t:slice} are actually slice. Moreover, assuming this theorem we prove Corollary~\ref{c:cor} which deals with the slice-ribbon conjecture for $3$-stranded pretzel knots. We start with the following proposition which explains an easy algorithm to determine whether a pretzel knot is ribbon.

\begin{prop}[Ribbon Algorithm]\label{p:alg}
Let $K=P(p_1,...,p_n)$ be a pretzel knot and let $p_{n+1}:=p_1$. While for some $j\in\{1,...,n\}$ it holds $p_j=-p_{j+1}$, we reduce the number of parameters to $n-2$ and repeat with the the knot $P(p_1,...,p_{j-1},p_{j+2},...,p_n)$. If at the end of the sequence of reductions we are left with a pretzel knot with exactly one parameter or with two parameters $a$ and $b$ satisfying $a=-b-1$, then $K$ is ribbon. 
\end{prop}
\begin{proof}
On a pretzel knot $K$, whenever there are two adjacent strands $p_1$ and $p_2$ with the same number of crossings but of opposite signs, we can perform the ribbon move shown in Figure~\ref{f:cob_pretzel}, which simplifies the pretzel knot yielding the disjoint union of an unknot and a new pretzel knot $K'$. The knot $K'$ is equal to $K$ without $p_1$ and $p_2$. Therefore, if $n$ is odd and after the sequence of reductions the set of parameters defining $K$ consists of only one integer, we have that after performing $\frac{n-1}{2}$ ribbon moves on $K$ we obtain the disjoint union of $\frac{n+1}{2}$ unknots. Thus, $K$ is ribbon. On the other hand, if $n$ is even and after the sequence of reductions the set of parameters defining $K$ consists of exactly two integers $a$ and $b$ satisfying $b=-a-1$, then after performing $\frac{n}{2}-1$ ribbon moves on $K$, we obtain, since $P(a,-a-1)$ is the unknot, the disjoint union of $\frac{n}{2}$ unknots. Thus again, $K$ is ribbon.
\end{proof}

\begin{cor}\label{c:affirm}
Let $K=P(p_{1},\dots,p_{n})$ be a pretzel knot satisfying the assumptions of Theorem~\ref{t:slice}. Then the above Ribbon Algorithm shows that for certain orderings of the parameters $K$ is slice.
\end{cor}

\begin{proof}[Proof of Corollary~\ref{c:cor}]
Given $K=P(p_1,p_2,p_3)$ as in the assumptions, if any of $p_1,p_2$ or $p_3$ equals $\pm 1$, then $K$ is a $2$--bridge knot (see \cite[Figure~2]{b:GJ} for a proof) and by \cite[Corollary~1.3]{b:Li}, the slice--ribbon conjecture holds in this case.  
On the other hand, if $|p_i|\geq 2$ for every $i\in\{1,2,3\}$, then the parameters satisfy either $p_1,p_2,p_3\equiv$(mod 2) or there is exactly one even parameter. For the first possibility \cite[Theorem~1.1]{b:GJ} holds and the statement follows. In the second case, if $K=P(p_1,p_2,p_3)$ is slice then Theorem~\ref{t:slice} holds and we obtain that $(p_1,p_2,p_3)$ is of the form $(p,-q,q)$ for some ordering. Since $3$--stranded pretzel knots are independent from the ordering of the parameters, we have $K=P(p,q,-q)$ and Proposition~\ref{p:alg} shows that $K$ is ribbon. 
\end{proof}

\begin{figure}
\begin{center}
\includegraphics[scale=0.7]{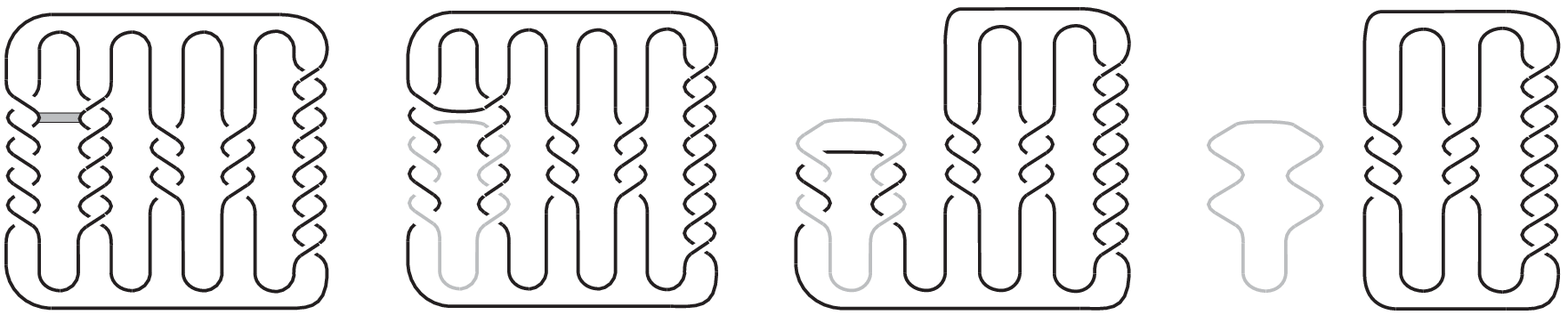}
\hcaption{On a pretzel knot, whenever there are two adjacent strands with the same number of crossings but of opposite sign, we can perform the ribbon move shown in this example obtaining the disjoint union of an unknot and the original pretzel knot with two strands less.}
\label{f:cob_pretzel}
\end{center}
\end{figure}
\begin{rem}
The only $3$-stranded pretzel knots for which the slice-ribbon conjecture remains open are a subset of the family $P(a,-a-2,-\frac{(a+1)^{2}}{2})$. We conjecture that none of these knots are slice and thus that the slice-ribbon conjecture holds for all $3$-stranded pretzel knots.
\end{rem}
Given a pretzel knot $P(p_1,...,p_n)$ with $|p_i|>1$ for all $i$, the necessary condition for sliceness that we establish in this paper, namely that for some permutation $\sigma$ of the parameters the knot $P(p_{\sigma(1)},...,p_{\sigma(n)})$ can be shown to be ribbon using Proposition~\ref{p:alg}, is not sufficient. For instance, in \cite[Section 11]{b:H} it is shown that the mutant $K_1=P(3,5,-3,-5,7)$ of the slice pretzel knot $K_2=P(3,-3,5,-5,7)$ is not slice. Therefore, one may ask  what constraints on the ordering of the parameters $(p_1,...,p_n)$ need to be added in order to obtain  a sufficient condition for sliceness. For all the examples we know (including in particular the knots $K_1$ and $K_2$ above) the ribbon algorithm of Proposition~\ref{p:alg} establishes that the knot $K(p_1,...,p_n)$ is actually ribbon. On the basis of these considerations we propose:
\begin{conj}\label{c:conj} 
The pretzel knots $P(p_1,...,p_n)$ with $|p_i|>1$ for all $i$ that are ribbon are precisely those detected by the algorithm in Proposition~\ref{p:alg}.
\end{conj}

\section{Preliminaries}\label{s:pre}

\subsection{Double branched covers of pretzel knots}

Let $\Gamma$ be a \textbf{plumbing graph}, that is, a graph in which every vertex $v_i$ carries an integer weight $a_i$, $i=1,...,n$. Associated to each vertex $v_i$ is the $4$--dimensional disc bundle $X\rightarrow S^2$ with Euler number $a_i$. If the vertex $v_i$ has $d_i$ edges connected to it in the graph $\Gamma$, we choose $d_i$ disjoint discs in the base of $X\rightarrow S^2$ and call the disc bundle over the $j$th disc $B_{ij}=D^2\times D^2$. When two vertices are connected by an edge, we identify $B_{ij}$ with $B_{k\ell}$ by exchanging the base and fiber coordinates and smoothing the corners. This pasting operation is called \textbf{plumbing} (for a more general treatment we refer the reader to \cite{b:GS}), and the resulting smooth $4$--manifold $M_\Gamma$ is said to be obtained by plumbing according to $\Gamma$. 

The group $H_2(M_\Gamma;\Z)$ has a natural basis represented by the zero-sections of the plumbed bundles. We note that all these sections are embedded $2$--spheres, and they can be oriented in such a way that the intersection form of $M_\Gamma$ will be given by the matrix $Q_\Gamma=(q_{ij})_{i,j=1,...,n}$ with the entries
$$q_{ij}
=
\left\{ 
\begin{array}{cl}
a_i & \mbox{if } i=j;\\
1 & \mbox{if } i \mbox{ is connected to } j \mbox{ by an edge;}\\
0 & \mbox{otherwise}.
\end{array}
\right.
$$ 
We will call $(\Z^n,Q_\Gamma)$ the intersection lattice associated to $\Gamma$.

A \textbf{star--shaped} graph  is a connected tree with a distinguished vertex $v_0$ (called the central vertex) such that the degree of any vertex other than the central one is $\leq 2$. A \textbf{leg} of a star--shaped graph is any connected component of the graph obtained by removing the central vertex. If $\Gamma$ is a star--shaped graph then the boundary $Y_{\Gamma}:=\partial M_{\Gamma}$ is a Seifert space (see \cite{b:Ran} for a proof) with as many singular fibers as legs of the graph $\Gamma$.

Given a pretzel knot $K=P(p_1,...,p_n)$, let $Y(p_{1},...,p_{n})$ denote the $3$--manifold obtained as the $2$--fold cover of $S^3$ branched along $K$. These $3$--manifolds are Seifert fibered spaces with $n$ singular fibers \cite{b:Mo}, which can be described as the boundary of the $4$--manifold obtained by plumbing according to the graph in Figure \ref{f:general}.
\begin{figure}
\begin{center}
\frag{0}{$0$}
\frag{3}{$p_1$}
\frag{5}{$p_n$}
\frag{7}{$p_2$}
\includegraphics[scale=0.9]{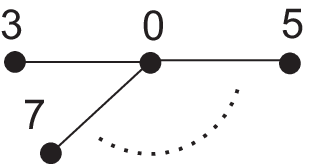}
\ccaption{Plumbing graph of a $4$-manifold with boundary $Y(p_{1},\dots,p_{n})$.}
\label{f:general}
\end{center}
\end{figure}
The order of the first homology group of $Y(p_{1},\dots,p_{n})$ can be computed via the incidence matrix of any graph $\Gamma$ such that $Y_{\Gamma}=Y(p_{1},\dots,p_{n})=\partial M_{\Gamma}$ and it holds (see \cite{b:NR})
\begin{equation}\label{e:det}
|H_{1}(Y_{\Gamma})|=|\det Q_{\Gamma}|=(\sum_{i=1}^{n}\frac{1}{p_{i}})\prod_{i=1}^{n}p_{i}.
\end{equation}
By \cite[Theorem~5.2]{b:NR}, the Seifert space $Y(p_{1},...,p_{n})$ can be written as the boundary of a negative definite $4$--plumbing as long as 
\begin{equation}\label{e:neg}
\frac{1}{p_{1}}+\cdots+\frac{1}{p_{n}}> 0.
\end{equation}
If the inequality holds, there is a \textbf{canonical negative plumbing tree}, which from now on will be denoted $\Psi$, satisfying $Y(p_{1},...,p_{n})=\partial M_\Psi$ and $M_{\Psi}$ is negative definite. All the vertices in $\Psi$ have weight $\leq -2$ except for the central vertex which has weight $\leq -1$. The tree $\Psi$ is obtained as follows: take the graph in Figure~\ref{f:general}: for every $p_i$ such that $p_{i}>1$ substitute its corresponding length--one leg with a $(-2)$--chain with $p_{i}-1$ vertices and subtract $1$ from the weight of the central vertex. In this way, we obtain a new four manifold, which is negative definite, and has the same boundary as before the substitutions. Formally this is done by a series of blow downs and blow ups (see \cite{b:Ne}). An example is shown in Figure~\ref{f:camb_signo}. We call $C_{i}$ the $(-2)$--chain corresponding to the parameter $p_{i}>1$.
\begin{figure}[h]
\begin{center}
\frag{m2}{$-2$}
\frag{m3}{$-4$}
\frag{m4}{$-2$}
\frag{m7}{$-7$}
\frag{m9}{$-9$}
\frag{0}{$0$}
\frag{3}{$3$}
\frag{5}{$5$}
\frag{4}{$4$}
\frag{2}{$2$}
\frag{1}{$1$}
\includegraphics[scale=0.7]{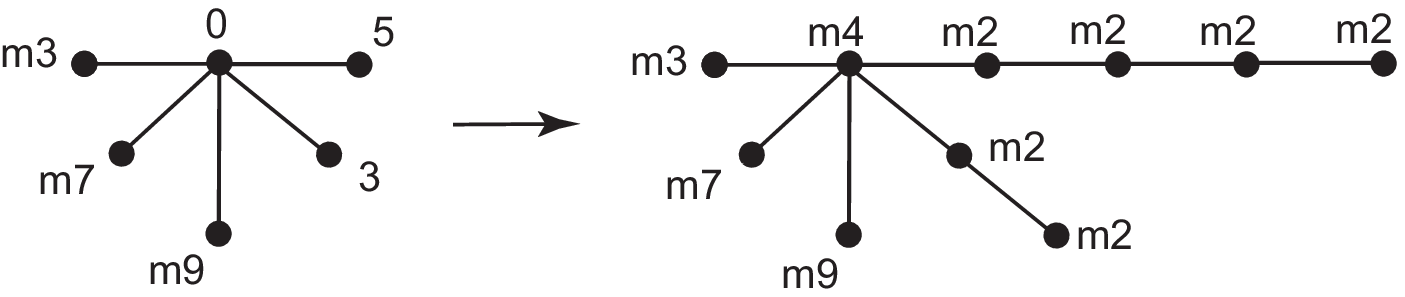}
\hcaption{After a series of blow ups/downs we achieve the canonical negative plumbing of the pretzel knot $P(-4,-7,-9,3,5)$.}
\label{f:camb_signo}
\end{center}
\end{figure}

Recall that if a knot $K$ is slice then its mirror image $\overline K$ is also slice. In the case of pretzel knots we have that for $K=P(p_{1},\dots,p_{n})$ the mirror image satisfies $\overline K=P(-p_{1},\dots,-p_{n})$. Therefore, when studying sliceness of pretzel knots, up to taking mirror images, we can always suppose that the double branched cover is the boundary of a negative definite $4$--plumbing or equivalently that the defining parameters satisfy inequality~\eqref{e:neg}. From now on we will only consider pretzel knots satisfying \eqref{e:neg} and we shall divide them in the following three families
\begin{itemize}
\item[\textbf{(p1)}] $n$ is even and all except one of the $p_{i}$ are odd.
\item[\textbf{(p2)}] $n$ is odd, all except one of the $p_{i}$ are odd and the only even parameter is positive.
\item[\textbf{(p3)}] $n$ is odd, all except one of the $p_{i}$ are odd and the only even parameter is negative.
\end{itemize}

Since the Seifert space $Y=Y(p_{1},...,p_{n})$ does not depend on the order of $p_{1},...,p_{n}$, from now on we adopt the following convention for the ordering and notation of the parameters. We write
$$Y=Y(a_1,...,a_s;c_1,...,c_t),\quad s,t\geq 0,\quad n=s+t,$$
where $a_1,...,a_s<-1$ and $c_1,...,c_t>1$. Note that the central vertex in $\Psi$ has weight $-t$.

We label the vertices of the graph $\Psi$ as follows: the central vertex will be called $v_0$; the vertices corresponding to the negative parameters $a_1,...,a_s$ will be called $v_1,...,v_s$; the vertices of the $(-2)$--chain $C_k$, $k\in\{1,...,t\}$, will be called $v_{1,k},...,v_{c_k-1,k}$, where $v_0$ is connected to $v_{1,k}$ and $v_{j,k}$ is connected to $v_{j+1,k}$ for all $j\in\{1,...,t-1\}$. The number of vertices in $\Psi$, which will be called $m$, coincides with the rank of $H_2(M_{\Psi};\Z)$. It is immediate to check that
$$m=|\Psi|=s+1+\sum_{i=1}^{t}(c_i-1).$$
Let $(\Z^{m},-\mathrm{Id})$ be the standard negative diagonal lattice with the $m$ elements of a fixed basis $\catE$ labeled as $\{e_{j}^k\}_{j,k}$.  As an abbreviation in notation let us write $e_j^k\cdot e_i^\ell$ to denote $-\mathrm{Id}(e_j^k,e_i^\ell)$. If the intersection lattice $(\Z^m,Q_{\Psi})$ admits an embedding $\iota$ into $(\Z^m,-\mathrm{Id})$, then we will omit the $\iota$ in the notation, i.e.\ instead of writing $\iota (v)=\sum_{j,k} x_j^ke_j^k$ we will directly write $v=\sum_{j,k} x_j^ke_j^k$. If $\iota$ exists we will call, for every $S\subseteq \Psi$,
$$U_S:=\{e_j^k\in\catE\,|\, e_j^k\cdot v\neq 0\mbox{ for some }v\in S\}.$$

\subsection{Signature of pretzel knots}\label{s:muu}

Let $K=P(p_1,...,p_n)$ be a pretzel knot and let $Y_\Gamma=\partial M_\Gamma$ be its double branched cover described as the boundary of a $4$--dimensional plumbing manifold. Since $K$ is a knot the determinant of the intersection form $Q_\Gamma$ is odd and the equation
$$Q_{\Gamma}(w,x)\equiv Q_{\Gamma}(x,x)\,(\mathrm{mod}\ 2)\qquad \forall x\in H_2(M_{\Gamma};\Z)$$
has exactly one solution in $H_2(M_\Gamma;\Z_2)$. This solution admits a unique integral lift $w\in H_2(M_\Gamma;\Z)$ such that its coordinates are $0$ or $1$ in the natural basis of $H_2(M_{\Gamma};\Z)$ given by the vertices $v_{1},...,v_{m}$ of the graph $\Gamma$. The homology class $w$ is called the  \textbf{Wu class}. There is a well defined subset $J\subset\{1,...,m\}$ such that
$$w=\sum_{j\in J} v_{j}\in H_2(M_\Gamma;\Z)$$
and we define the \textbf{Wu set} as $\{v_j\in\Gamma\,|\,j\in J\}$. In order to calculate the Wu set for a given plumbing graph $\Gamma$ with odd determinant, we apply the algorithm described in \cite[Theorem~7.1]{b:NR}: start by reducing the graph $\Gamma$ to a collection (possibly empty) of isolated points with odd weights by a sequence of moves of type $1$ and $2$ below. Consider a leaf $v\in\Gamma$ connected to the vertex $u\in\Gamma$.
\begin{itemize}
\item Move 1: If the weight on $v$ is even, then erase $v$ and $u$ from $\Gamma$.
\item Move 2: If the weight of $v$ is odd, then erase $v$ and change the parity of the weight on $u$.
\end{itemize}
In order to determine which vertices belong to the Wu--set we undo the sequence of movements starting with the isolated vertices until we reobtain $\Gamma$, taking the following into account:
\begin{itemize}
\item All the isolated vertices with odd weight obtained in the final step of the reduction of $\Gamma$ belong to the Wu set.
\item If we undo Move $1$, then the vertex $u$ does not belong to the Wu set whereas the vertex $v$ will belong to the Wu set only if the weight on the vertex $u$ and the number of adjacent vertices to $u$ which already belong to the Wu set do not have the same parity. 
\item If we undo Move $2$, then the vertex $v$ will belong to the Wu set if and only if $u$ does not belong to the Wu set.
\end{itemize}
In Figure~\ref{f:Wuset} the encircled vertices form the Wu--set of the canonical negative plumbing graphs associated to pretzel knots in families \textbf{(p1)}, \textbf{(p2)} and \textbf{(p3)}. In family \textbf{(p2)} we assume that the only even parameter is $c_t$, while in family \textbf{(p3)} the only even parameter is $a_1$.
\begin{figure}
\begin{center}
\frag[ss]{Ct1}{$C_{1}$}
\frag[ss]{Ct2}{$C_{2}$}
\frag[ss]{Ctn}{$C_t$}
\frag[ss]{a1ae}{$a_{1}\equiv 0$ (mod 2)}
\frag[ss]{Ctnb}{$C_t,\ c_{t}\equiv 0$ (mod 2)}
\frag{p2}{\textbf{(p1)}}
\frag{p3}{\textbf{(p3)}}
\frag{p2b}{\textbf{(p1)}}
\frag{p3b}{\textbf{(p2)}}
\frag[ss]{m2}{$-2$}
\frag[ss]{nt}{$\!-t$}
\frag[ss]{nt1}{$\!-t$}
\frag[ss]{a1}{$a_{1}$}
\frag[ss]{a2}{$a_{2}$}
\frag[ss]{at}{$a_{s}$}
\includegraphics[scale=0.55]{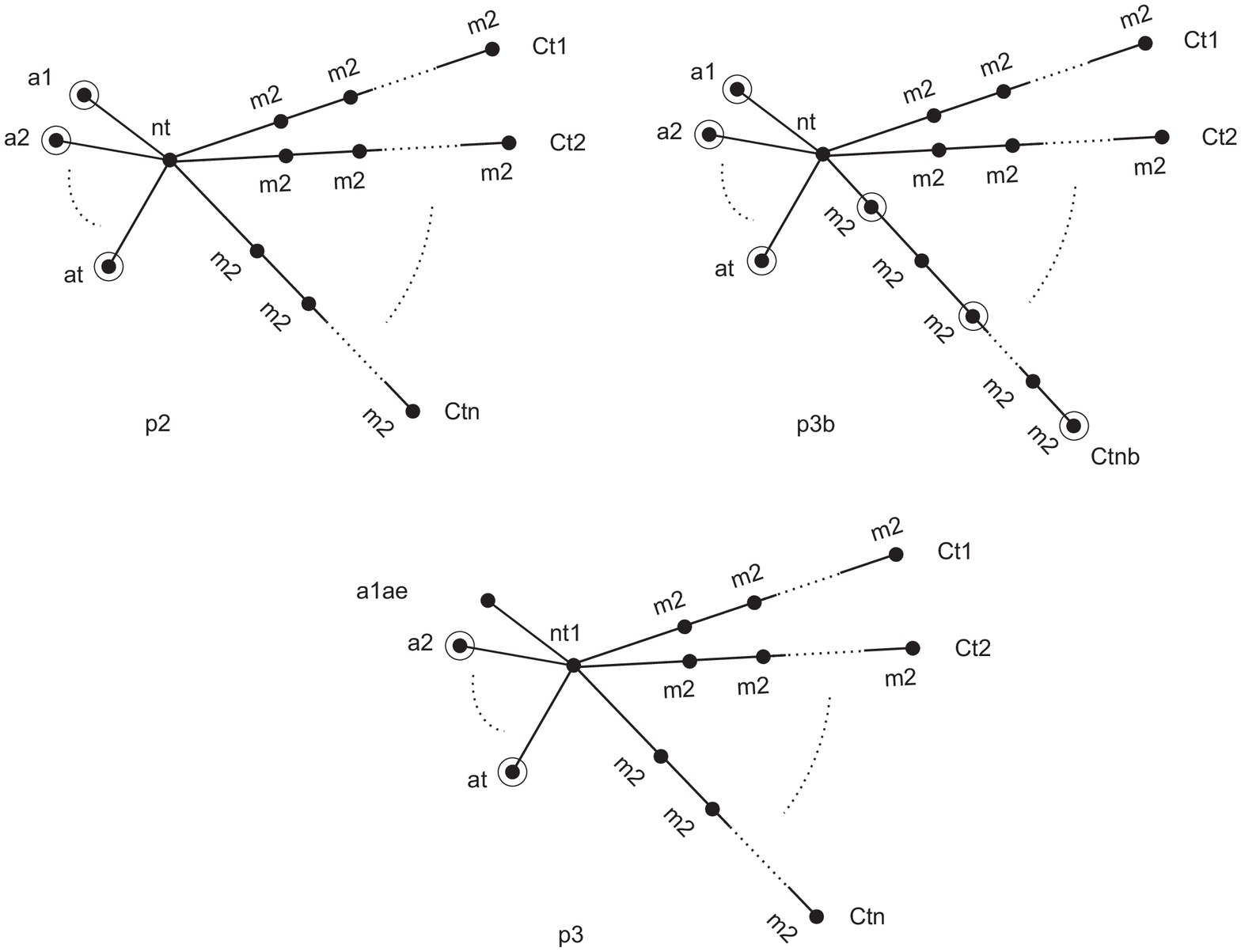}
\hcaption{Wu--set on canonical negative plumbing graphs corresponding to the pretzel knots in families {\textbf{(p1)}}, {\textbf{(p2)}} and {\textbf{(p3)}}.}
\label{f:Wuset}
\end{center}
\end{figure}

The following formula, due to Saveliev \cite[Theorem~5]{b:Sa}, expresses the signature of the pretzel knot $P(p_1,...,p_n)$ as
\begin{align}\label{e:Sa}
\sigma(P(p_1,...,p_n))=\mathrm{sign}(Q_{\Gamma})-w\cdot w,
\end{align}
where $w\cdot w$ stands for $Q_{\Gamma}(w,w)$. Notice that the expression $\mathrm{sign}(Q_{\Gamma})-w\cdot w$ equals $\bar\mu(Y_\Gamma)$, where $\bar\mu$ is Neumann's invariant. It is well known, \cite[Theorem~8.3]{b:Ka}, that slice knots have vanishing signature. In Section~\ref{s:1step}, we will find constraints on the parameters defining the Seifert spaces $Y$ which arise as double branched covers over slice pretzel knots combining equality \eqref{e:Sa} with Donaldson's theorem on the intersection form of definite $4$-manifolds. 

\section{Two interesting families of pretzel knots}\label{s:two}
In this section we study the sliceness of two subfamilies of pretzel knots of type \textbf{(p3)}. Both of them stem from the detailed general study on pretzel knots developed in the remaining sections but need specific arguments to show that the knots in these families are not slice. In the first subsection we study the family $P(a,-a-2,-a-\frac{a^2+9}{2})$ with $a\geq 3$ odd. We shall prove that the double branched covers of these knots do not bound rational homology balls, in spite of the fact that the obstruction given by the $d$-invariants from Heegaard-Floer homology (see Section~\ref{s:d}) vanishes. The main tool we use are the Casson-Gordon invariants.
The second subsection deals with the family $P(a,-a-2,-\frac{(a+1)^{2}}{2})$ with $a\geq 3$ odd. A major difficulty presented by this family is that the double branched covers of these knots are all integer homology spheres. As further explained in Section~\ref{s:weird} many of the recent obstructions to knot sliceness defined from Heegard-Floer homology or Khovanov homology vanish for this family while, perhaps surprisingly, the Alexander polynomial is able to detect the non-sliceness of many (perhaps all) the knots in this family.

\subsection{Casson-Gordon invariants and the family $P(a,-a-2,-a-\frac{a^2+9}{2})$ }\label{s:step3}

In 1975 Casson and Gordon introduced some knot invariants that allowed to show that not all algebraically slice knots are smoothly slice. The invariants depend on a knot $K\subset S^{3}$ and on the choice of a character $\chi$ defined on the first homology group of the double branched cover of $K$. In \cite{b:CG} two different invariants, denoted $\sigma(K,\chi)$ and $\tau(K,\chi)$, are defined from the difference of the twisted signatures of some 4-manifolds associated to the couple $(K,\chi)$. We shall only deal with the properties of a specific version of $\sigma(K,\chi)$ that suits our purposes.

Consider a slice knot $K\subset S^{3}$ with double branched cover $Y$. Let $p$ be a prime, $r\in\N$ and $\chi:H_{1}(Y;\Z)\rightarrow\Z_{p^{r}}$ be a character of order $p^{r}$. Suppose that the covering $\widetilde Y$ induced by the character satisfies $H_{1}(\widetilde Y;\Q)=0$. Let $W_{D}$ be the double cover of $B^{4}$ branched over a slicing disc for $K$ and $V$ be the kernel of the map $i_{*}:H_{1}(Y;\Z)\rightarrow H_{1}(W_{D};\Z)$ induced by the inclusion.

\begin{thm}[Casson-Gordon]\label{t:CG} 
With the above assumptions, for every character $\chi$ of prime power order vanishing on $V$ we have $|\sigma(K,\chi)|\leq 1$.
\end{thm}

We shall compute the Casson-Gordon invariants of some of the pretzel knots in family \textbf{(p3)} via the formula given in \cite[Theorem~6.7]{b:CiF}. This formula computes $\sigma(K,\chi)$ from a surgery presentation of $Y$ regarded as a colored link and the term $\sigma_{L}(\omega)$ stands for the colored signature of $L$.

\begin{thm}[Cimasoni-Florens]\label{t:CF}
Let $Y$ be the 3-manifold obtained by surgery on a framed link $L$ with $m$ components and linking matrix $Q$. Let $\chi:H_{1}(M;\Z)\rightarrow\Z_{p^{r}}$ be the character mapping the meridian $\mu_{i}$ of the $i$-th component of $L$ to $n_{i}$ with $1\leq n_{i}<p^{r}$ and $n_{i}$ coprime to $p$. Consider $L$ as an $m$-colored link and set $\omega=(n_{1},\dots,n_{m})$. Then,
\begin{equation}\label{formula}
\sigma(K,\chi)=\sigma_L(\omega)-\sum_{i<j}Q_{i,j}-\mathrm{sign}\,Q+\frac{2}{p^{2r}}\sum_{i,j}(p^{r}-n_i)n_jQ_{ij}
\end{equation}
\end{thm}

In order to use this formula to compute the Casson-Gordon invariants for the pretzel knots $K_{a}=P(a,-a-2,-a-\frac{a^2+9}{2})$, $a>1$ odd, we will use the surgery presentation of the Seifert space $Y_{a}=Y(a,a-2,-a-\frac{a^2+9}{2})$ given by the framed link $L_{a}$ in Figure~\ref{f:surgery}. It is obtained from the diagram associated to the canonical negative plumbing tree by blowing down the central vertex and subsequently blowing down every new $-1$-framed unknot.
\begin{figure}
\centering
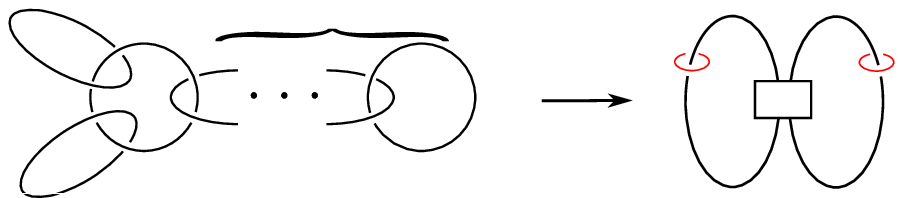
\hcaption{The diagram on the left corresponds to the negative definite plumbing manifold bounded by the Seifert space $Y_{a}=Y(a,-a-2,-a-\tfrac{a^2+9}{2})$. After a series of blow downs we obtain the diagram on the right that we shall use to compute the Casson-Gordon invariants. In the diagram $L_{a}$ the two meridians, generators of $H_{1}(S^{3}\setminus L_{a};\Z)$, are depicted in red and the box labelled $a$ stands for $a$ right handed full crossings between the strands entering the box.}
\label{f:surgery}
\end{figure}
From this surgery presentation of $Y_{a}$ we can read the following presentation of its first homology group:
$$H_{1}(Y_{a};\Z)=\langle\mu_{1},\mu_{2}|-\tfrac{a^{2}+9}{2}\mu_{1}+a\mu_{2},a\mu_{1}-2\mu_{2}\rangle=\langle\mu_{1},\mu_{2}|\tfrac{a+9}{2}\mu_{1}-\mu_{2},9\mu_{1}\rangle.$$
It follows that $H_{1}(Y_{a};\Z)\cong\Z_{9}$ with $\mu_{1}$ being a generator of the group and $\mu_{2}=\tfrac{a+9}{2}\mu_{1}$. We now proceed to define a character on $H_{1}(Y_{a};\Z)$ vanishing on the subgroup $V$, which is cyclic of order $3$ generated by $3\mu_{1}$ (cf. Remark~\ref{r:order}). The character $\chi:H_{1}(Y_{a};\Z)\rightarrow\Z_{3}$ such that $\chi(\mu_{1})=1$, where $1$ is a generator of $\Z_{3}$, has the desired properties.

We have all the necessary ingredients to compute $\sigma(K_{a},\chi)$ via the formula~\eqref{formula}. Notice however that if $a\equiv 0$(mod 3) then the above setting yields $\chi(\mu_{2})=0$ which is not a valid value for $n_{2}$ in the hypothesis of Theorem~\ref{t:CF}. Therefore we shall first compute $\sigma(K_{a},\chi)$ for $a\not\equiv 0$(mod 3) and later we will use yet another surgery presentation of $Y_{a}$ to deal with this remaining case.

The linking matrix for the link $L_{a}$ in Figure~\ref{f:surgery} is given by
$$
Q_{a}=
\cbra{
\begin{matrix}
-\tfrac{a^{2}+9}{2} & a \\
a & -2
\end{matrix}
}
$$
and formula~\eqref{formula} yields
\begin{align}
\sigma(K_{a},\chi)= & \sigma_{L_{a}}(1,n_{2})-a+2+\frac{2}{9}\left(-\tfrac{a^{2}+9}{2}2+a2n_{2}+a(3-n_{2})-2(3-n_{2})n_{2})\right)\notag\\
\label{calculo}
=& \sigma_{L_{a}}(1,n_{2})-a+2+\frac{2}{9}\left(-a^{2}-13+a(n_{2}+3)\right),
\end{align}
where $n_{2}$ has two possible values, $1$ and $2$, depending on whether $a\equiv 2$(mod 3) or $a\equiv 1$(mod 3). 

Our aim is to show that the knots $K_{a}$ are not slice and to do so it suffices to estimate $|\sigma(K_{a},\chi)|$ in \eqref{calculo} and show it is greater than 1. Indeed, all the assumptions of Theorem~\ref{t:CG} are satisfied since $H_{1}(Y_{a};\Z)$ is cyclic and the character $\chi$ is of order $3$, which implies that the first rational homology of the induced covering vanishes \cite[Lemma~4.4]{b:CG2}. The term $\sigma_{L_{a}}(1,n_{2})$ is the coloured signature of the link $L_{a}$ and it is defined as the signature of a $n\times n$ hermitian matrix where $n$ is the rank of the first homology group of a $C$-complex for $L_{a}$ (see \cite{b:CiF} for the definitions and details). The $C$-complex for $L_{a}$ depicted in Figure~\ref{f:Ccomplex} allows us to determine $\sigma_{L_{a}}(1,n_{2})\in\{-a+1,a-1\}$. 
\begin{figure}
\centering
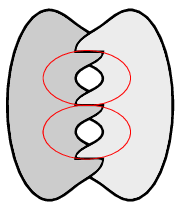
\hcaption{This figure presents a $C$-complex for $L_{3}$ with a basis of its first homology depicted in red. The evident generalisation to the links $L_{a}$ shows that the order of the first homology of the $C$-complex is $a-1$ and therefore $\sigma_{L_{a}}\in\{-a+1,a-1\}$.}
\label{f:Ccomplex}
\end{figure}
Since the term $\left(-a^{2}-13+a(n_{2}+3)\right)$ is negative for all $a$ and $\sigma_{L_{a}}(1,n_{2})-a+2\leq 1$ we have the following estimate
\begin{equation*}
|\sigma(K_{a},\chi)|\geq \frac{2}{9}\left(a^{2}+13-a(n_{2}+3)\right)-1>1\quad\mbox{for all }a>1\mbox{ odd},\ a\not\equiv 0\mbox{(mod 3).}
\end{equation*}
This last inequality shows that the knots $K_{a}$ with $a\not\equiv 0$(mod 3) are not slice. 

To deal with the remaining case, the knots $K_{a}$ with $a\equiv 0$(mod 3), we shall use the surgery presentation of $Y_{a}$ given in Figure~\ref{f:alternativa}. 
\begin{figure}
\centering
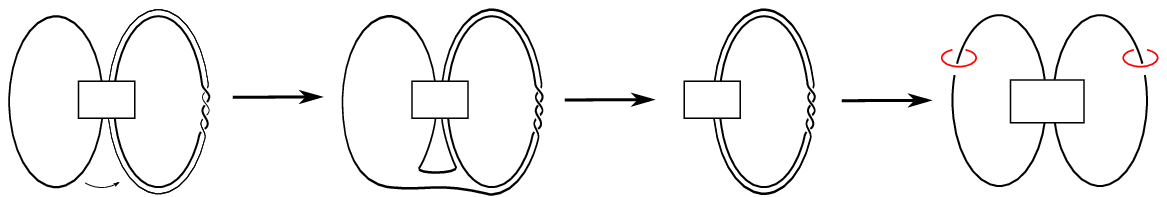
\hcaption{Starting with the link $L_{a}$, in which the thin curve represents the framing of the right component, we preform the indicated handle slide to obtain the second diagram. After a series of isotopies we obtain the link $\tilde L_{a}$, which is a suitable surgery presentation of $Y_{a}$ to compute the Casson-Gordon invariants of $K_{a}$ when $a\equiv 0$(mod 3). Again the red curves in the last diagram are the meridians generating $H_{1}(S^{3}\setminus\tilde L_{a};\Z)$.}
\label{f:alternativa}
\end{figure}
It is obtained from the diagram in Figure~\ref{f:surgery} by sliding the left handle over the handle with framing $-2$. The new framings are obtained applying the rules of Kirby calculus (see \cite{b:GS} for details). This surgery diagram yields the following presentation of the first homology group of $Y_{a}$:
$$H_{1}(Y_{a};\Z)=\langle\mu_{1},\mu_{2}|-\tfrac{(a+2)^{2}+9}{2}\mu_{1}-(a+2)\mu_{2},-(a+2)\mu_{1}-2\mu_{2}\rangle=\langle\mu_{1},\mu_{2}|\tfrac{a-7}{2}\mu_{1}+\mu_{2},9\mu_{1}\rangle.$$
As before $H_{1}(Y_{a};\Z)\cong\Z_{9}$ is generated by $\mu_{1}$ and this time we have $\mu_{2}=\tfrac{7-a}{2}\mu_{1}$. We choose the character $\chi:H_{1}(Y_{a};\Z)\rightarrow\Z_{3}$ defined by $\chi({\mu_{1}})=1$ which satisfies the assumptions of Theorem~\ref{t:CG} and since $\chi(\mu_{2})=2$ we can use the formula~\eqref{formula} to compute $\sigma(K_{a},\chi)$. This time the link $\tilde L_{a}$ of the surgery presentation in Figure~\ref{f:alternativa} has linking matrix
$$
Q_{a}=
\cbra{
\begin{matrix}
-\tfrac{(a+2)^{2}+9}{2} & -a-2 \\
-a-2 & -2
\end{matrix}
}
$$
and the formula reads
\begin{align*}\label{calculo2}
\sigma(K_{a},\chi)= & \sigma_{\tilde L_{a}}(1,2)+a+2+2+\frac{2}{9}\left(-\tfrac{(a+2)^{2}+9}{2}2-(a+2)4-(a+2)-4\right)=\\
& \sigma_{\tilde L_{a}}(1,2)+a+4-\frac{2}{9}\left(27+9a+a^{2})\right).
\end{align*}
Similar arguments to the ones used before allow us to estimate
$$
|\sigma(K_{a},\chi)|\geq -2a-5+6+2a+\tfrac{2a^{2}}{9}=\tfrac{2a^{2}}{9}+1>1
$$
and therefore by Theorem~\ref{t:CG} we have the following statement. 

\begin{thm}
For all odd $a>1$ the knots in the family $P(a,-a-2,-a-\frac{a^2+9}{2})$ are not slice. 
\end{thm}

\begin{rem}\label{r:ratballs}
In fact we have shown something stronger than the non sliceness of the pretzel knots in the family $P(a,-a-2,-a-\frac{a^2+9}{2})$: the arguments in the proof of Theorem~\ref{t:CG} imply that the Seifert spaces $Y_{a}=Y(a,-a-2,-a-\frac{a^2+9}{2})$ do not bound rational homology balls. 
\end{rem}

\subsection{Alexander polynomials and the family $P(a,-a-2,-\frac{(a+1)^{2}}{2})$}\label{s:weird}
This section is devoted to the study of the sliceness of the pretzel knots of the form $P_{a}=P(a,-a-2,-\frac{(a+1)^{2}}{2})$ with $a\geq 3$ odd. All the knots in this family have determinant $1$ and therefore the double branched covers $Y_{a}$ are integer homology spheres. It follows that the Casson-Gordon invariants cannot be used to study the existence of rational homology balls bounded by the Seifert manifolds $Y_{a}$. We will pursue the study of the sliceness of the knots $P_{a}$ leaving open the question of the existence of rational homology balls bounded by their double branched covers.

There are several well-known obstructions to sliceness that we have computed for the knots of the form $P_{a}$ but all of them vanished. In the sequel we shall not use any of the following facts but we have decided to include them for completeness. In the remaining sections we will show that each knot $P_{a}$ has vanishing signature, its determinant is a square and Donaldson's theorem does not obstruct sliceness. Moreover, we have checked that the only $d$-invariant of their double branched covers, which are  homology spheres, vanishes. The hat version of the knot Floer homology of pretzel knots is known. The family $P_{a}$ lies within the hypothesis of \cite[Theorem 2]{b:Ef} which combined with the Alexander polynomials computed below  suffices to determine that the Ozsv\'ath-Szab\'o $\tau$ invariant is zero for all the knots in this family. Moreover, the Rasmussen $s$-invariant of the knot $P(3,-5,-4)=12_{n_475}$ is known to be zero and a crossing change argument implies that the first knot in our family, namely the knot $P(3,-5,-8)$, also has vanishing Rasmussen invariant. 

Given all this vanishing of obstructions one might be tempted to think that the knots in the family $P_{a}$ are actually slice. However, we conjecture that this is not the case for any parameter $a$ and we will show that indeed for $a\not\equiv 1,11,37,47,49,59$ (mod 60) the knot $P_{a}$ does not bound a disk embedded in the $4$-ball. The invariant capable of detecting the non-sliceness of these knots is the Alexander polynomial, in its most classical version. Perhaps it is surprising that one of the first invariants used in the slice problem is still capable of distinguishing subtleties that more modern invariants fail to see. 

We have postponed the tedious computation of the Alexander polynomials $\Delta_{a}(t)$ of the knots $P_{a}$ to Appendix A. In it we show that it holds
\begin{align*}
\Delta_{a}(t)&\doteq \frac{t^{a+2}+1}{t+1}\frac{t^a+1}{t+1}-\frac{(a+1)^{2}}{4}t^{a-1}(t-1)^{2}\\
&=\prod_{
\substack{d|a+2\\ d\neq 1}
}\Phi_{d}(-t)
\prod_{
\substack{\delta|a\\\delta\neq 1}
}\Phi_{\delta}(-t)-\frac{(a+1)^{2}}{4}t^{a-1}(t-1)^{2},
\end{align*}
where $\Phi_{n}$ stands for the $n$-th cyclotomic polynomial.

If the knots in family $P_{a}$ were slice then it would follow, by Fox and Milnor's theorem \cite{b:FM}, that $\Delta_{a}(t)\doteq f(t)f(t^{-1})$ for some polynomial $f(t)$. Our goal is to show that this is not the case and to this end several strategies are possible. If we showed that the polynomials $\Delta_{a}(t)$ are irreducible in $\Z[t]$, we would be done. Stepan Orevkov has kindly taken a look at this problem and informed us that he checked with the computer and up to $a=1597$ the polynomials $\Delta_{a}(t)$ are in fact irreducible. However we have not found a proof to show the irreducibility of every polynomial in our family. Another possible approach is to look at these polynomials modulo $p$ for some prime and study their irreducibility in $\F_{p}[t]$. However, by \cite[Theorem 12]{b:AV} the number of irreducible factors of $\Delta_{a}(t)$ (mod $p$) is even for every odd $p$, which implies that many (perhaps all) the polynomials $\Delta_{a}(t)$ are irreducible over the integers but reducible over $\F_{p}$ for every $p$. 

We are not able to show that none of the polynomials $\Delta_{a}(t)$ satisfies the Fox-Milnor factorisation but there is a lot of evidence pointing that this might be the case. Taking advantage of the rich literature on the reducibility of cyclotomic polynomials, we will prove the following statement.

\begin{thm}\label{t:Alex} 
For $a\not\equiv 1,11,37,47,49,59$ (mod 60) the polynomials $\Delta_{a}(t)$ do not have a Fox-Milnor factorisation. 
\end{thm}

\begin{cor}
For $a\not\equiv 1,11,37,47,49,59$ (mod 60) the pretzel knots $P_{a}$ are not slice.
\end{cor}

\begin{proof}[Proof of Theorem~\ref{t:Alex}]
Since the Alexander polynomial $\Delta_{a}(t)$ is a \emph{self reciprocal} polynomial, i.e.\ $\Delta_{a}(t)=t^{\deg\Delta_{a}}\Delta_{a}(t^{-1})$, its irreducible factors are all self reciprocal or come in \emph{reciprocal pairs}, that is if $g|\Delta_{a}$ and $g$ is not self reciprocal then $g^{*}(t):=t^{\deg g}g(t^{-1})$ is also a factor of $\Delta_{a}$. Suppose that there exists a polynomial $f$ such that $\Delta_{a}(t)=t^{\deg\Delta_{a}}f(t)f(t^{-1})$ then we have that the irreducible self reciprocal factors of $\Delta_{a}$ all have even multiplicity. The idea of the proof is to show that the reduction mod $p$, for suitable primes $p$, of $\Delta_{a}$ has an odd number of self reciprocal irreducible factors. We start with a prime number $p$ dividing $\frac{(a+1)^{2}}{4}$ and consider the polynomial 
\begin{equation}\label{e:reduction}
\overline{\Delta}_{a}^{p}(t):=\Delta_{a}(t)\ (\mathrm{mod}\ p)=\prod_{
\substack{d|a+2\\ d\neq 1}
}\overline{\Phi}_{d}^{p}(-t)
\prod_{
\substack{\delta|a\\\delta\neq 1}
}\overline{\Phi}_{\delta}^{p}(-t).
\end{equation}
Notice that since gcd$(p,a)=$gcd$(p,a+2)=$gcd$(a,a+2)=1$ all the irreducible factors of the polynomials $\overline{\Phi}_{d}^{p}$ and $\overline{\Phi}_{\delta}^{p}$ in $\F_{p}[t]$ appearing in \eqref{e:reduction} have multiplicity one and are all distinct. Since cyclotomic polynomials are self reciprocal, if $\Delta_{a}$ satisfies the Fox-Milnor condition, it follows that every cyclotomic polynomial in \eqref{e:reduction} has an even number of irreducible factors. It is well known that the number of irreducible factors of $\overline{\Phi}_{d}^{p}$ equals the quotient between $\varphi(d)$ and the order of $p$ mod $d$, where $\varphi$ is Euler's totient function. Let us call this quotient $N_{d}^{p}$. It follows that if $\Delta_{a}$ has a Fox-Milnor factorisation then for every $p$ dividing $\frac{(a+1)^{2}}{4}$ and every $d|a+2$, $d\neq 1$, and every $\delta|a$, $\delta\neq 1$, the numbers  $N_{d}^{p}$ and $N_{\delta}^{p}$  are even. The rest of the proof will consist of suitable choices of $p$ and $a$ that force $N_{d}^{p}$ or $N_{\delta}^{p}$ to be odd.

Let us consider the odd parameter $a$ modulo 12. If $a\equiv 3,7$ (mod 12) then $3$ divides either $a$ or $a+2$ and $\frac{(a+1)^{2}}{4}$ is even. We obtain $N_{3}^{2}=1$. Moreover, if $a\equiv 9$ (mod 12) then $3$ divides $a$ and we choose $p$ dividing $\frac{(a+1)^{2}}{4}$ such that $p\equiv 2$ (mod 3). This choice is always possible since under these assumptions $\frac{(a+1)^{2}}{4}=(6n+5)^{2}$, $n\in\N$, and $6n+5\equiv 2$ (mod 3). Again in this case we obtain $N_{3}^{p}=1$. It follows that for $a\equiv 3,7,9$ (mod 12) the polynomial $\Delta_{a}$ does not have a Fox-Milnor factorisation.

We can push further this same argument to study the cases $a\equiv 1,11$ (mod 12) with less success. If $a=12n+11$ for some $n$  then $\frac{(a+1)^{2}}{4}$ is even and we choose $p=2$. The difficulty comes in the choice of $d$, since in this case there is not a common divisor for all $a$ or $a+2$. Adding the hypothesis $n\equiv 1,2$ (mod 5) we obtain that $d=5$ divides $a$ or $a+2$ and $N_{5}^{2}=1$. Similarly, if $a=12n+1$ and $n\equiv 1,2$ (mod 5) then again $d=5$ divides either $a$ or $a+2$. We choose as $p$ dividing $\frac{(a+1)^{2}}{4}=(6n+1)^{2}$ a prime such that $p\equiv 2,3$ (mod 5). This choice is always possible since under the current assumptions $6n+1\equiv 2,3$ (mod 5). In this case it follows that $N_{5}^{p}=1$. Summing up we have shown that for $a\equiv 13,23,25,35$ (mod 60) the polynomial $\Delta_{a}$ does not factor as $t^{\deg\Delta_{a}}f(t)f(t^{-1})$.

The last case we shall study is $a=12n+5$ with a different argument. We shall prove that in this case the polynomial $t^{2}+t+1$ is a reducible factor of the mod 2 reduction of $\Delta_{a}$ and it has multiplicity 1. Instead of working with the previous normalisation of the Alexander polynomial we shall use the Laurent polynomial $t^{-\frac{\deg\Delta_{a}}{2}}\Delta_{a}$ which we will keep calling $\Delta_{a}$ to ease the notation. The irreducible factor we shall be looking at is then $t+1+t^{-1}$. It is not difficult to check that we have the following identities
\begin{align*}
\overline{\Delta}^{2}_{5}&=t^{5}+t^{3}+1+t^{-3}+t^{-5},\\
\overline{\Delta}^{2}_{12n+5}(t)&=\overline{\Delta}^{2}_{12(n-1)+5}(t)+\sum_{i=1}^{6}(t^{12(n-1)+5+2i}+t^{-(12(n-1)+5+2i)}).
\end{align*}
Since these polynomials verify, in $\F_{2}$,
\begin{small}
\begin{align*}
&t^{5}+t^{3}+1+t^{-3}+t^{-5}\equiv (t+1+t^{-1})(t^{4}+t^{3}+t^{2}+1+t^{-2}+t^{-3}+t^{-4})& (\mathrm{mod}\ 2),\\
&\sum_{i=1}^{6}(t^{x+2i}+t^{-(x+2i)})\equiv(t+1+t^{-1})(\sum_{i=0}^{2}(t^{x+11-i}+t^{-(x+11-i)})+(\sum_{i=0}^{2}(t^{x+5-i}+t^{-(x+5-i)})& (\mathrm{mod}\ 2),
\end{align*}
\end{small}%
an easy induction argument shows that $t+1+t^{-1}$ is an irreducible factor of $\overline{\Delta}^{2}_{12n+5}$ for all $n\in\N$. However $(t+1+t^{-1})^{2}\equiv t^{2}+1+t^{-2}$ (mod 2) is never a factor of $\overline{\Delta}^{2}_{12n+5}$. Indeed, one can easily check that $t^{2}+1+t^{-2}$ does not divide $\overline{\Delta}^{2}_{5}$, but on the other hand we have
\begin{equation*}
\sum_{i=1}^{6}(t^{x+2i}+t^{-(x+2i)})=(t^{2}+1+t^{-2})(t^{x+10}+t^{x+4}+t^{-(x+4)}+t^{-(x+10)}).
\end{equation*}
Again and easy induction argument yields the claim and the theorem is proved.
\end{proof}

\begin{rem}
The first knot in the family $P_{a}$ that is not covered by Theorem~\ref{t:Alex} is $P_{11}=P(11,-13,-72)$. However it is immediate to check that in this case $N_{11}^{2}=1$ and therefore $P_{11}$ is not slice. It is frustrating to accept the fact that for every single knot we have checked we have found an argument to determine its non-sliceness but that all attempts to generalise the proof to the whole family have miserably failed.
\end{rem}

\section{The general case: first obstructions to sliceness}\label{s:1step}

We start now the study of slice pretzel knots with one even parameter in full generality. A necessary condition for a pretzel knot $K:=P(p_1,...,p_n)$ to be slice is that the intersection lattice associated to its canonical negative plumbing graph $(\Z^m,Q_\Psi)$  admits an embedding into the standard diagonal negative lattice of the same rank. This follows, as explained in the introduction, from the fact that if $K$ is slice, the associated Seifert space $Y$ smoothly bounds a rational homology ball. As remarked in Section~\ref{s:pre} up to considering mirror images we can suppose that $\sum_{i=1}^n\frac{1}{p_i}>0$ and thus the canonical negative plumbing graph exists. Moreover, $K$ being slice implies that the invariant $\sigma(K)$ vanishes. In this section, imposing these two conditions, the existence of the embedding and the vanishing of the knot signature, we get important information on $p_1,...,p_n$ and on the embedding $(\Z^m,Q_\Psi)\hookrightarrow(\Z^m,-\mathrm{Id})$. The conclusions differ depending on whether the parameters satisfy the conditions \textbf{(p1)}, \textbf{(p2)} or \textbf{(p3)} defined in Section~\ref{s:pre}. In fact, the existence of the embedding and the condition $\sigma(K)=0$ totally determine $Y$ for  the families \textbf{(p1)} and \textbf{(p2)} and they determine $Y$ up to one parameter for the family \textbf{(p3)}.

Given $\Psi$, we are interested in whether an embedding into the standard negative lattice exists, however the particular embedding is not relevant. With the notation established in Section~\ref{s:pre} in the following lemma we prove that, up to a change of basis, the images under the embedding of the $(-2)$--chains and of the central vertex are totally determined.

\begin{lem}\label{l:-2c}
Let $Y:=Y(a_1,...,a_s;c_1,...,c_t)$ be a Seifert manifold, boundary of a negative definite canonical plumbing $M_\Psi$, whose associated intersection lattice $(\Z^{m},Q_{\Psi})$ admits an embedding into $(\Z^{m},-\mathrm{Id})$. Moreover, suppose that the associated pretzel link is a knot. Then, up to a change of basis,
\begin{itemize}
\item[$(1)$] For every $k\in\{1,...,t\}$, the image of the vertices $v_{1,k},...,v_{c_k-1,k}$ of the $(-2)$--chain $C_k$ is $v_{i,k}=e_i^{k}-e_{i+1}^k$, for every $i\in\{1,...,c_k-1\}$. In particular, $U_{C_k}\cap U_{C_\ell}=\emptyset$ for every $k,\ell\in\{1,...,t\}$ with $k\neq\ell$.
\item[$(2)$] The central vertex satisfies $v_0=-e_1^{1}-e_1^{2}-...-e_1^t$, with $e_1^{k}$ defined in $(1)$.
\item[$(3)$]$s\geq t-1$.
\end{itemize}
\end{lem}
\begin{proof}
We start by proving $(1)$ using the same argument as in \cite{b:GJ}. The only linear combinations of the $\{e_i^k\}_{i,k}$ which lead vectors of square $-2$ are of the form $\pm e_i^k\pm e_j^\ell$. Since $v_{1,k}$ has square $-2$, up to reindexing the basis of $\Z^m$ and up to scaling by $-1$, we must have $v_{1,k}=e_1^k-e_2^k$. If $c_{k}>2$ then there exists $v_{2,k}$, which has also square $-2$, and since $v_{1,k}\cdot v_{2,k}=1$, one of the vectors appearing in $v_{2,k}$ must be either $-e_1^k$ or $e_2^k$. Thus, again up to reindexing and rescaling, we are forced to define $v_{2,k}=e_2^k-e_3^k$. So far we have proved $(1)$ for $c_k=2$ or $c_k=3$.
In the proof of \cite[Lemma~3.1]{b:GJ} it is explained how to proceed by induction to show that (up to a change of basis) we are forced to make the assignment $v_{i,k}=e_i^k-e_{i+1}^k$ for all $i\in\{1,...,c_k-1\}$, provided $c_k\geq 5$. Therefore, we are left with the case $c_k=4$. It is easy to check that this time there is a second possibility for the embedding of the $(-2)$--chain $C_k$, namely
\begin{align*}
v_{1,k}&=e_1^k-e_2^k,\\
v_{2,k}&=e_2^k-e_3^k,\\
v_{3,k}&=-e_1^k-e_2^k.
\end{align*}
However, since $v_{1,k}\cdot v_0=1$ we must have $\{e_1^k,e_2^k\}\cap U_{v_0}\neq\emptyset$, which is not compatible with $v_0\cdot v_{3,k}=0$.

In order to conclude with $(1)$ suppose by contradiction that for two different indices $k,\ell\in\{1,...,t\}$ there exist $i,j$ and $e_s^r$ such that $e_s^r\in U_{v_{i,k}}\cap U_{v_{j,\ell}}$. Since $v_{i,k}\cdot v_{j,\ell}=0$, there must exist some other basis vector, say $e_u^x$, in $U_{v_{i,k}}\cap U_{v_{j,\ell}}$ and since $v_{i,k}\cdot v_{i,k}=v_{j,\ell}\cdot v_{j,\ell}=-2$, we have $U_{v_{i,k}}=U_{v_{j,\ell}}=\{e_s^r,e_u^x\}$. By assumption $P(a_1,...,a_s,c_1,...,c_t)$ is a knot, which implies in particular that at most one parameter among $c_1,...,c_t$ is  equal to $2$ and therefore, at least one of the chains $C_k$ and $C_\ell$ has length strictly greater than one, say it is $C_k$. It follows from the above arguments that exactly one between $e_s^r$ and $e_u^x$ belongs to either $U_{v_{i-1,k}}$ or $U_{v_{i+1,k}}$. Let us fix, without loss of generality, that $e_s^r\in U_{v_{i+1,k}}$, then we have that $v_{i+1,k}\cdot v_{j,\ell}\neq 0$. This contradiction proves $(1)$.

Next, all the $(-2)$--chains $C_{1},...,C_t$ are connected to the central vertex and therefore $|U_{v_0}\cap U_{C_k}|\geq 1$ for every $k$. Since there are exactly $t$ $(-2)$--chains and the weight of the central vertex is $-t$, we have $|U_{v_0}\cap U_{C_k}|=1$ for every $k$. Hence, up to eventually reindexing and scaling by $-1$, we have that for every $k\in\{1,...,t\}$ the vector $-e_1^{k}$ is a summand in the expression of $v_0$. Thus, $(2)$ is proved.   

We are left with $(3)$, which follows from a straightforward computation. In fact, from $(1)$ we know that for each $k\in\{1,...,t\}$, the image of the chain $C_k$ is contained in the span of $c_k$ vectors of the basis $\{e_i^k\}_{i,k}$. Moreover, for every $k,\ell\in\{1,...,t\}$, $k\neq\ell$, the images of $C_k$ and $C_\ell$ are disjoint. Therefore,  the rank of the image of the embedding, which in the statement is called $m$, must be at least equal to $\sum_{k=1}^t c_k$. On the other hand, being $\Psi$ the canonical negative plumbing graph, we know that 
$$m=s+1+\sum_{k=1}^t (c_k-1).$$
Hence, $s\geq t-1$ and $(3)$ holds.
\end{proof}

Once established in Lemma~\ref{l:-2c} the convention on the embedding of the $(-2)$--chains and the central vertex, we study in Lemmas~\ref{l:p1}, \ref{l:p2} and \ref{l:p3} the embedding of the whole lattice associated to $\Psi$. In the three lemmas we will assume the following slice conditions:
\begin{itemize}
\item[(SC)]$n=s+t\geq 3$, $Y:=Y(a_1,...,a_s;c_1,...,c_t)$ is a Seifert manifold, boundary of a negative definite canonical plumbing $M_\Psi$, whose associated intersection lattice $(\Z^{m},Q_{\Psi})$ admits an embedding into $(\Z^{m},-\mathrm{Id})$ and $\sigma(P_\Psi)=0$, where $P_\Psi$ is the associated pretzel knot. 
\end{itemize}

\begin{rem}
Notice that, with the notation that we have fixed for Seifert spaces, the condition \textbf{(p1)} reads $s+t$ even and among $a_1,...,a_s,c_1,...,c_t$ there is exactly one even integer; the condition \textbf{(p2)} asks $s+t$ odd and all $a_1,...,a_s,c_1,...,c_t$ odd except for one $c_i$; finally, \textbf{(p3)} forces $s+t$ odd and all $a_1,...,a_s,c_1,...,c_t$ odd except for one $a_i$.
\end{rem}

\begin{lem}\label{l:p1}
Assume {\normalfont(SC)} and that $a_1,...,a_s,c_1,...,c_t$ satisfy {\normalfont\textbf{(p1)}}. Then $s=t$,
$$Y=Y(-c_{1},-c_{2},...,-c_{t}-1;c_{1},...,c_{t}).$$
and the embedding into the diagonal standard negative lattice is, up to a change of basis, the one in Figure~\ref{f:p2}.
\end{lem}
\begin{figure}[h]
\begin{center}
\frag[ss]{-e1tn-e1tn-e1n}{$-\!e_{1}^{1}\!-\!e_{1}^{2}\!-\!...\!-\!e_{1}^{t}$}
\frag[ss]{e1t1}{$e_{1}^{1}\!-\!e_{2}^{1}$}
\frag[ss]{e2}{$e_{2}^{1}\!-\!e_{3}^{1}$}
\frag[ss]{e1a}{$e_{c_{1}\!-\!1}^{1}\!-e_{c_{1}}^{1}$}
\frag[ss]{e1}{$e_{1}^{2}\!-\!e_{2}^{2}$}
\frag[ss]{e3}{$e_{2}^{2}\!-\!e_{3}^{2}$}
\frag[ss]{e4}{$e_{c_{2}-1}^{2}\!-\!e_{c_{2}}^{2}$}
\frag[ss]{e5}{$e_{1}^{t}\!-\!e_{2}^t$}
\frag[ss]{e6}{$e_{2}^{t}\!-\!e_{3}^{t}$}
\frag[ss]{h}{$ $}
\frag[ss]{f}{${\color{grigio}e_1^0+e_2^0}$}
\frag[ss]{e7}{$e_{c_{t}-1}^{t}\!-\!e_{c_{t}}^{t}$}
\frag[ss]{a1}{$e_{1}^{1}+e_{2}^{1}+...+e_{c_{1}}^{1}$}
\frag[ss]{a2}{$e_{1}^{2}+e_{2}^{2}+...+e_{c_{2}}^{2}$}
\frag[ss]{a3}{$e_{1}^{t-1}+e_{2}^{t-1}+...+e_{c_{t-1}}^{t-1}$}
\frag[ss]{a4}{$e_{1}^{t}+e_{2}^{t}+...+e_{c_{t}}^{t}+e_1^0$}
\includegraphics[scale=0.55]{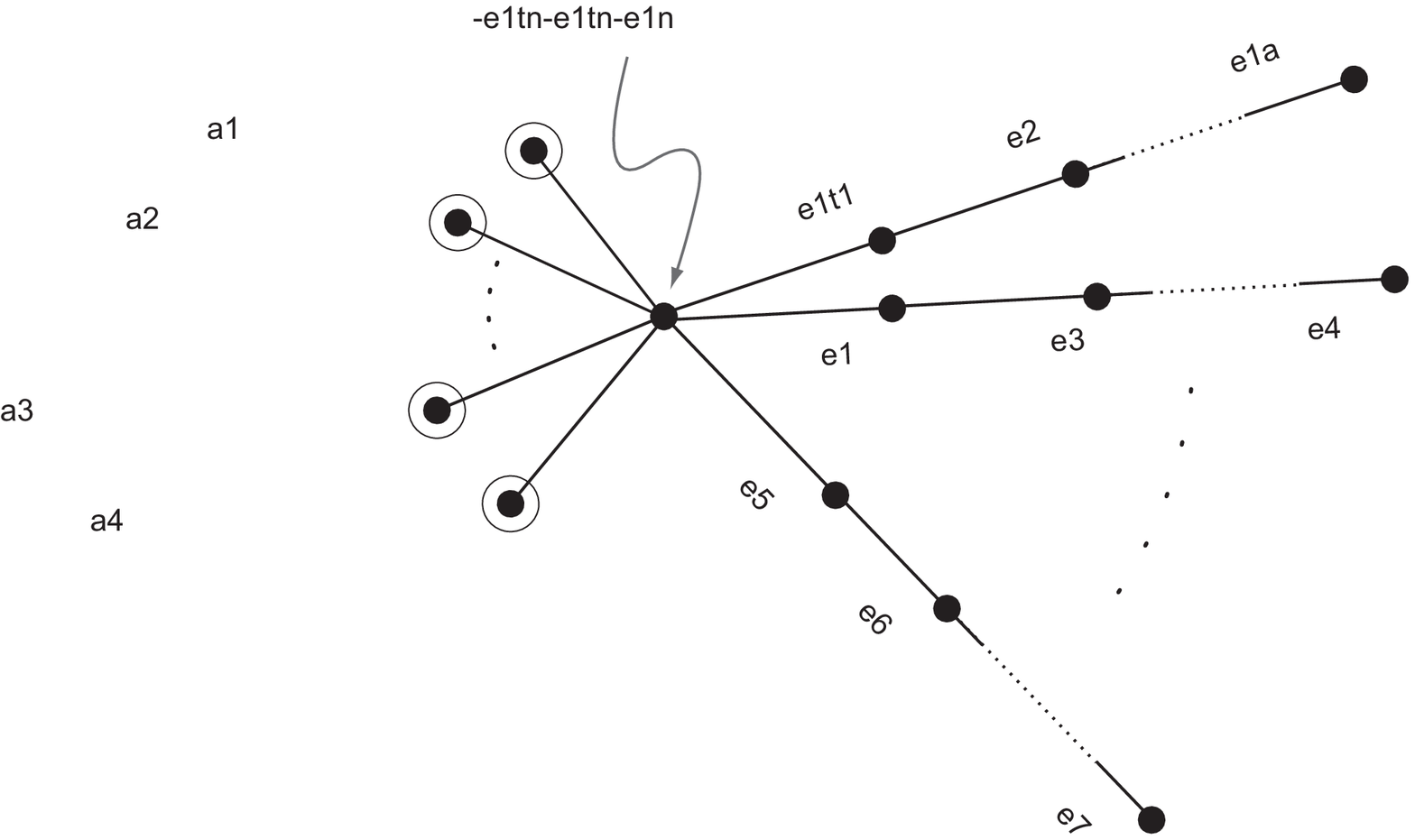}
\hcaption{Up to a change of basis the embedding of the canonical negative plumbing graph for family \textbf{(p1)}. The basis of $\Z^{m}$ is $\{e^0_1,e_{j}^{k}\}$, where $k\in\{1,...,t\}$ and $j\in\{1,...,c_{k}\}$.}
\label{f:p2}
\end{center}
\end{figure}

\begin{proof}
The homology Wu class $w$, depicted in Figure~\ref{f:Wuset}, is $w=\sum_{i=1}^{s}v_{i}$ and hence $w\cdot v_{i}=a_{i}$ for all $i\in\{1,...,s\}$. The assumption $\sigma(P_\Psi)=0$ together with the fact that $M_\Psi$ is negative definite imply, by \eqref{e:Sa}, that $w\cdot w=-m$. By definition of $w$, its embedding into $(\Z^m,-\mathrm{Id})$ must be of the form $w=\sum_{j,k}\beta_{j}^k e_{j}^k$ with $\beta_{j}^k\in\{\pm 1\}$. Let us write the embedding of the $v_i$ as $v_i=\sum_{j,k}x(i)_j^ke_j^k$, where $x(i)_j^k\in\Z$. The equation $w\cdot v_{i}=v_{i}\cdot v_{i}$ implies 
$$-\sum_{j,k}\beta_j^kx(i)_j^k=-\sum_{j,k}(x(i)_j^k)^2.$$
Thus, $\beta_j^kx(i)_j^k\in\{0,1\}$ and for every $i$, the vector $v_{i}$ is a linear combination of exactly $|a_{i}|$ vectors of the basis $\{e_j^k\}_{j,k}$ with coefficients $\pm 1$. Moreover, since for every $i,j\in\{1,...,s\}$ with $i\neq j$, we have $v_{i}\cdot v_{j}=0$ from the equality $w=\sum_{i=1}^{s}v_{i}=\sum_{j,k}\beta_{j}^k e_{j}^k$ we deduce easily that $U_{v_{i}}\cap U_{v_{j}}=\emptyset$. Furthermore, it follows that for each $(-2)$--chain $C_{k}$, $k\in\{1,...,t\}$, there must exist one and only one $i(k)\in\{1,...,s\}$ such that $U_{v_{i(k)}}\cap U_{C_{k}}\neq\emptyset$. Notice that we do not yet exclude the possibility $i(k)=i(k')$ for $k\neq k'$. Since $C_{k}$ is orthogonal to $v_{i(k)}$, it follows  that $U_{C_{k}}\subseteq U_{v_i(k)}$, which implies in particular that $c_{k}\leq |a_{i(k)}|$ for every $k$. 

Consider the set $\Delta:=\{v_i\subset w|\,U_{v_i}\cap U_{C_k}=\emptyset\ \forall k\}$. Observe that $s\leq |\Delta| +t$. By Lemma~\ref{l:-2c}\,$(1)$, it holds that 
$$|U_{w\setminus\Delta}|\geq\sum_{k=1}^t c_k,$$
and therefore, since  $U_{v_i}\cap U_{v_j}=\emptyset$ for every $v_i,v_j\subset w$, $i\neq j$, we have  
$$|U_\Delta|\leq m-\sum_{k=1}^t c_k=s+1+\sum_{k=1}^{t}(c_{k}-1)-\sum_{k=1}^t c_k=s+1-t.$$
Since, for every $i,\ell\in\{1,...,s\}$ with $i\neq\ell$ it holds $a_i\leq -2$, $U_{v_{i}}\cap U_{v_{\ell}}=\emptyset$ and $|x(i)_j^k|\leq 1$ for every $j,k$, it follows that the embedding in the statement is possible only if $2|\Delta|\leq |U_\Delta|$. Hence, only if $s\leq t+1$. On the other hand, the embedding of the $(-2)$--chains requires, by Lemma~\ref{l:-2c}$(3)$, that $s\geq t-1$. Since, by assumption \textbf{(p1)}, we have $n=s+t\equiv 0\,$(mod 2) it follows that $s=t$. Furthermore, this last equality implies that $\Delta=\emptyset$ and that for each $i\in\{1,...,s\}$, there exists at most one $k\in\{1,...,t\}$ such that $U_{v_i}\cap U_{C_k}\neq\emptyset$. In fact, assume by contradiction that there exist $i,k,k'$, with $k\neq k'$, such that $U_{C_k\cup C_{k'}}\subseteq U_{v_i}$, then $|\Delta|\geq s-t+1$ and the inequality $2|\Delta|\leq |U_\Delta|$ gives $s\leq t-1$. This contradiction shows that if $k\neq k'$, then $i(k)\neq i(k')$ and also that $s\in\{t,t+1\}$. From now on, we assume without loss of generality that $i(k)=k$ for every $k$. 

Equality $s=t$ forces

$$1+\sum_{k=1}^{t}c_{k}=m=-\mathrm{sign}(Q_{\Psi})=\underbrace{-\sigma(P_\Psi)}_{=0}-w\cdot w=\sum_{k=1}^{t}|a_k|.$$

Since $c_{k}\leq |a_{k}|$, we must have $c_{k}=|a_k|$ for all parameters except one. Without loss of generality we fix that $a_k=-c_k$ for $k\in\{1,...,t-1\}$ and $a_t=-c_t-1$. Notice that, by assumption \textbf{(p1)}, either $a_t$ or $c_t$ is the only even parameter among $a_1,...,a_t,c_1,...,c_t$. At this point, after fixing the embedding of the $(-2)$--chains and the central vertex applying Lemma~\ref{l:-2c}, it is straightforward to check that the embedding is as claimed.
\end{proof}

\begin{lem}\label{l:p2}
Assume {\normalfont(SC)} and that $a_1,...,a_s,c_1,...,c_t$ satisfy {\normalfont\textbf{(p2)}}. Then $s=t-1$, 
$$Y=Y(-c_{1},...,-c_{t-1};c_{1},...,c_{t})$$
and the embedding into the diagonal negative lattice is, up to a change of basis, the one in Figure~\ref{f:p3a}.
\end{lem}

\begin{figure}[h]
\begin{center}
\frag[ss]{-e1tn-e1tn-e1n}{$-\!e_{1}^{1}\!-\!e_{1}^{2}\!-\!...\!-\!e_{1}^{t}$}
\frag[ss]{e1t1}{$e_{1}^{1}\!-\!e_{2}^{1}$}
\frag[ss]{e2}{$e_{2}^{1}\!-\!e_{3}^{1}$}
\frag[ss]{e1a}{$e_{c_{1}\!-\!1}^{1}\!-e_{c_{1}}^{1}$}
\frag[ss]{e1}{$e_{1}^{2}\!-\!e_{2}^{2}$}
\frag[ss]{e3}{$e_{2}^{2}\!-\!e_{3}^{2}$}
\frag[ss]{e4}{$e_{c_{2}-1}^{2}\!-\!e_{c_{2}}^{2}$}
\frag[ss]{e5}{$e_{1}^{t}\!-\!e_{2}^t$}
\frag[ss]{e6}{$e_{2}^{t}\!-\!e_{3}^{t}$}
\frag[ss]{e7}{$e_{c_{t}-1}^{t}\!-\!e_{c_{t}}^{t}$}
\frag[ss]{a1}{$e_{1}^{1}+e_{2}^{1}+...+e_{c_{1}}^{1}$}
\frag[ss]{a2}{$e_{1}^{2}+e_{2}^{2}+...+e_{c_{2}}^{2}$}
\frag[ss]{a3}{$e_{1}^{t-2}+e_{2}^{t-2}+...+e_{c_{t-2}}^{t-2}$}
\frag[ss]{a4}{$e_{1}^{t-1}+e_{2}^{t-1}+...+e_{c_{t-1}}^{t-1}$}
\frag[ss]{Ce}{$c_t\equiv 0$(mod 2)}
\includegraphics[scale=0.55]{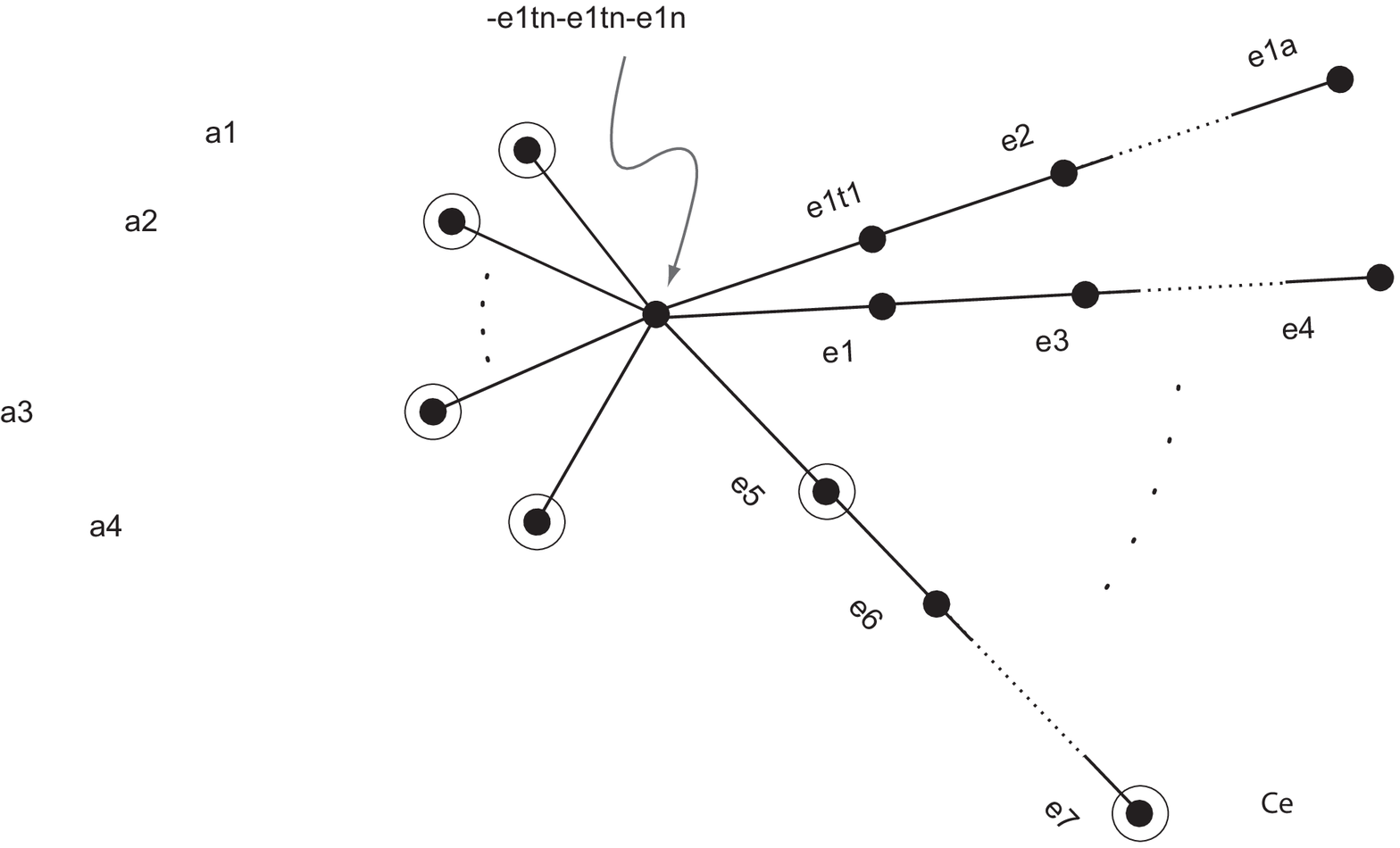}
\hcaption{Up to a change of basis the embedding of the canonical negative plumbing graph for family \textbf{(p2)},  where $c_t>0$ is the only even parameter. The basis of $\Z^{m}$ being $\{e_{j}^{k}\}$ where $k\in\{1,...,t\}$ and $j\in\{1,...,c_{k}\}$.}
\label{f:p3a}
\end{center}
\end{figure}

\begin{proof}
By assumption, the parameters defining $Y$ verify \textbf{(p2)} and therefore there is exactly one even parameter and moreover, it is positive. Without loss of generality, we fix it to be $c_t$. The homology Wu class $w$, depicted in Figure~\ref{f:Wuset}, is 
$$w=v_{1,t}+v_{3,t}+...+v_{c_t-1,t}+\sum_{i=1}^{s}v_{i},$$
where $v_{i}$ is the vertex in the graph $\Psi$ with weight $a_{i}$ and $v_{2j-1,t}\in C_{t}$ for all $j\in\{1,...,\frac{c_t}{2}\}$.

From the assumptions $\sigma(P_\Psi)=0$ and the existence of an embedding of the intersection lattice associated to $\Psi$ into the standard negative diagonal lattice of rank $m$, it follows, exactly as in the proof of Lemma~\ref{l:p1}, that for every $i,\ell\in\{1,...,s\}$ with $i\neq\ell$, the vector $v_{i}$ is the linear combination of exactly $|a_{i}|$ vectors from the basis $e_j^k$ with coefficients $\pm 1$ and that $U_{v_{i}}\cap U_{v_{\ell}}=\emptyset$. Since it holds $w=\sum_{j,k}\beta_{j}^k e_{j}^k$, where $\beta_{j}^k\in\{\pm 1\}$, it follows that $U_{C_{t}}\cap U_{v_{i}}=\emptyset$ for every $i$. Moreover, we deduce that for each $(-2)$--chain $C_{k}$, $k\in\{1,...,t-1\}$, there must exist one and only one $i(k)\in\{1,...,s\}$ such that $U_{v_{i(k)}}\cap U_{C_{k}}\neq\emptyset$. Notice that we do not yet exclude the possibility $i(k)=i(k')$ for $k\neq k'$. Since $C_{k}$ is orthogonal to $v_{i(k)}$, it follows  that $U_{C_{k}}\subseteq U_{v_i(k)}$, which implies in particular that $c_{k}\leq |a_{i(k)}|$ for every $k\in\{1,...,t-1\}$.

Following the proof of Lemma~\ref{l:p1} we define the set $\Delta:=\{v_i\subset w|\,U_{v_i}\cap U_{C_k}=\emptyset\ \forall k\in\{1,...,t-1\}\}$. Observe that this time we have $|\Delta|\geq s-t+1$. Arguing as in Lemma~\ref{l:p1} we obtain the condition $2|\Delta|\leq |U_\Delta|$, which gives $s\leq t-1$. On the other hand, the embedding of the $(-2)$--chains requires, by Lemma~\ref{l:-2c}$(3)$, that $s\geq t-1$. Moreover, $s=t-1$ implies that
for every $i\in\{1,..,s\}$, there exists at most one $k\in\{1,...,t-1\}$ such that $U_{v_i}\cap U_{C_k}\neq\emptyset$ (hence $\Delta=\emptyset$). Indeed, if this were not the case, then $|\Delta|\geq s-t+2$ and the inequality $2|\Delta|\leq |U_\Delta|$ would lead to the contradiction $s\leq t-3$. Therefore, from now on, we assume without loss of generality that $i(k)=k$ for every $k\in\{1,...,t-1\}$.

We have the following equalities.
\begin{align}\label{e:dif2}
\sum_{k=1}^{t-1}c_{k}+c_t=m=\underbrace{-\sigma(P_\Psi)}_{=0}-w\cdot w=-\sum_{k=1}^{t-1}a_{k}+c_t=\sum_{k=1}^{t-1}|a_{k}|+c_t.
\end{align}

Recall that, for every $k\in\{1,...,t-1\}$, it holds $c_k\leq |a_k|$. Then, equalities \eqref{e:dif2} force $c_{k}=|a_{k}|$ for every $k\in\{1,...,t-1\}$, while $c_t$, the only even parameter, has no constraints. In this way we have proved the relationship among the parameters $a_1,...,a_{t-1},c_1,...,c_t$ claimed in the statement. The embedding of the $(-2)$--chains and of the central vertex follows from Lemma~\ref{l:-2c} and at this point it is immediate to check that the rest of the embedding must be as suggested in Figure~\ref{f:p3a}.
\end{proof}

\begin{lem}\label{l:p3}
Assume {\normalfont(SC)} and that $a_1,...,a_s,c_1,...,c_t$ satisfy {\normalfont\textbf{(p3)}}. Then $s=t+1$ and $$Y=Y(-c_{1}-\lambda^2-(c_1-\lambda)^2,-c_{1}-2,-c_2,...,-c_{t};c_1,...,c_t)$$ 
for some $\lambda\in\Z$. Moreover, the embedding into the diagonal negative lattice is, up to a basis change, the one in Figure~\ref{f:p3b}.
\end{lem}
\begin{figure}[h]
\begin{center}
\frag[ss]{-e1tn-e1tn-e1n}{$-\!e_{1}^{1}\!-\!e_{1}^{2}\!-\!...\!-\!e_{1}^{t}$}
\frag[ss]{e1t1}{$e_{1}^{1}\!-\!e_{2}^{1}$}
\frag[ss]{e2}{$e_{2}^{1}\!-\!e_{3}^{1}$}
\frag[ss]{e1a}{$e_{c_{1}\!-\!1}^{1}\!-e_{c_{1}}^{1}$}
\frag[ss]{e1}{$e_{1}^{2}\!-\!e_{2}^{2}$}
\frag[ss]{e3}{$e_{2}^{2}\!-\!e_{3}^{2}$}
\frag[ss]{e4}{$e_{c_{2}-1}^{2}\!-\!e_{c_{2}}^{2}$}
\frag[ss]{e5}{$e_{1}^{t}\!-\!e_{2}^t$}
\frag[ss]{e6}{$e_{2}^{t}\!-\!e_{3}^{t}$}
\frag[ss]{e7}{$e_{c_{t}-1}^{t}\!-\!e_{c_{t}}^{t}$}
\frag[ss]{a1}{$e_{1}^{1}+e_{2}^{1}+...+e_{c_{1}}^{1}\!\!-\!\lambda e_{1}^0-(c_1\!-\!\lambda) e_{2}^0$}
\frag[ss]{a2}{$e_{1}^{1}+e_{2}^{1}+...+e_{c_{1}}^{1}+e_{1}^0+e_{2}^0$}
\frag[ss]{a3}{$e_{1}^{2}+e_{2}^{2}+...+e_{c_{2}}^{2}$}
\frag[ss]{a4}{$e_{1}^{t}+e_{2}^{t}+...+c_{c_t}^t$}
\includegraphics[scale=0.55]{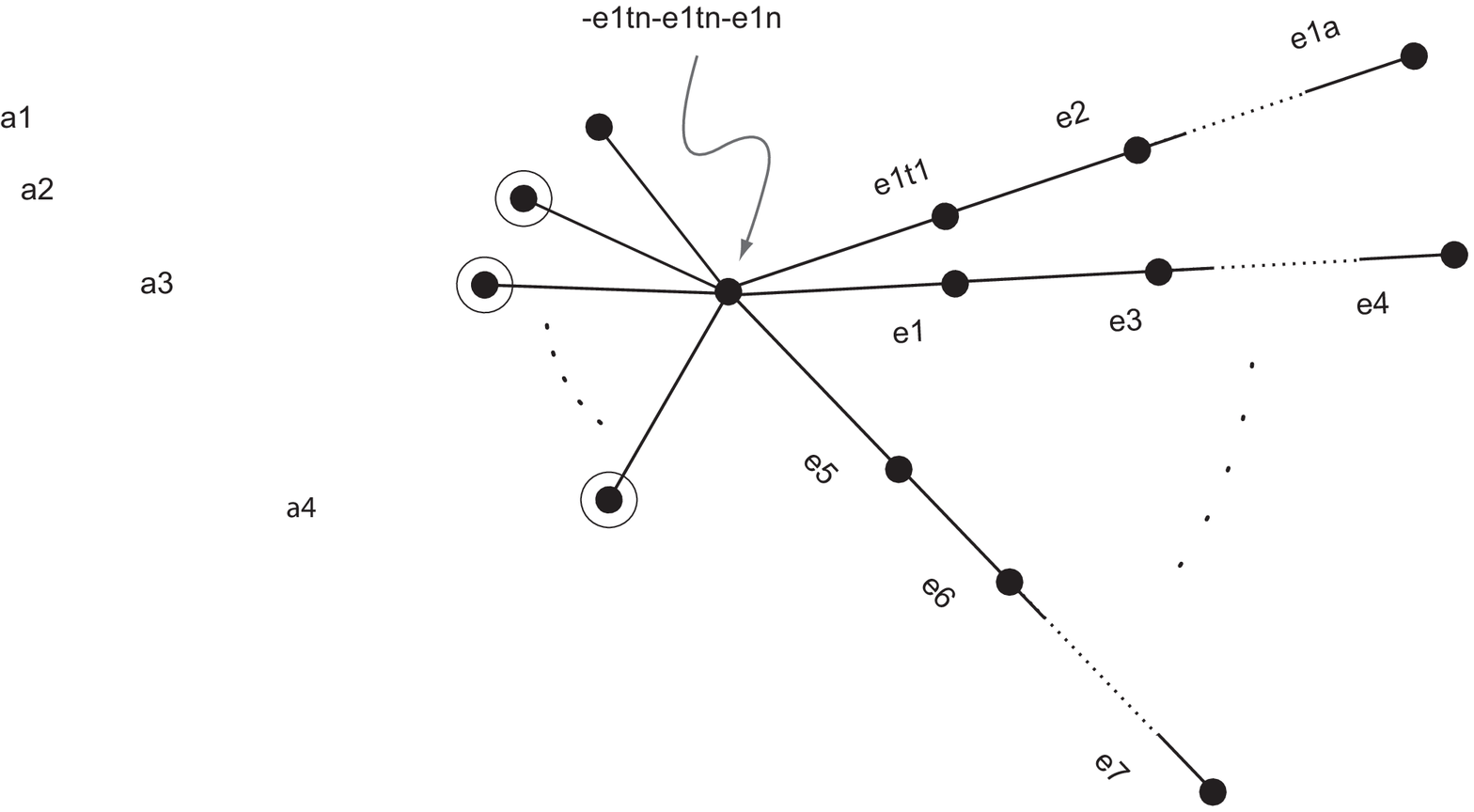}
\hcaption{Up to a change of basis the embedding of a plumbing graph in family \textbf{(p3)}. The basis of $\Z^{m}$ being $\{e^0_1,e^0_2,e_{j}^{k}\}$ where $k\in\{1,...,t\}$ and $j\in\{1,...,c_{k}\}$.}
\label{f:p3b}
\end{center}
\end{figure}

\begin{proof}
By assumption \textbf{(p3)}, among $a_1,...,a_s,c_1,...,c_t$ there is exactly one even parameter  and this parameter is negative. Without loss of generality we fix it to be $a_1$. The homology Wu class $w$, depicted in Figure~\ref{f:Wuset}, is $w=\sum_{i=2}^{t}v_{i}$. 

We start with the same argument of Lemma~\ref{l:p1}. From the assumptions $\sigma (P_\Psi)=0$ and the existence of an embedding of the intersection lattice associated to $\Psi$ into the standard negative diagonal lattice of rank $m$, it follows that for every $i,\ell\in\{2,...,s\}$ with $i\neq\ell$, $v_{i}$ is the linear combination of exactly $|a_{i}|$ vectors from the basis $e_j^k$ with coefficients $\pm 1$ and that $U_{v_{i}}\cap U_{v_{\ell}}=\emptyset$. 

Next, by Lemma~\ref{l:-2c}(3) the embedding of the $(-2)$--chains is possible only if $s\geq t-1$. On the other hand, being $\Psi$ the canonical negative plumbing graph, it holds $v_0\cdot v_0=-t$. This implies, since for every $i,j\in\{2,...,s\}$ with $i\neq j$ we have $v_i\cdot v_0=1$ and $U_{v_{i}}\cap U_{v_{j}}=\emptyset$, that $s-1\leq t$. Hence, $s\in\{t-1,t,t+1\}$. Moreover, by assumption \textbf{(p3)}, we know $s+t\equiv 1$(mod 2), which forces $s\in\{t-1,t+1\}$. 

We assume first $s=t-1$. In this case the Wu set of $\Psi$ consists of the $t-2$ vertices $v_2,...,v_{t-1}$. These vertices are connected to the central vertex, which has weight $-t$. Recall that for every $i,j\in\{2,...,t-1\}$ with $i\neq j$, we have $U_{v_i}\cap U_{v_j}=\emptyset$ and that $m=|U_w|$. It follows that either
\begin{itemize}
\item[$(i)$]there exist two different indices $i,j\in\{2,...,t-1\}$ such that $|U_{v_i}\cap U_{v_0}|=|U_{v_j}\cap U_{v_0}|=2$; or
\item[$(ii)$] there exists exactly one $i\in\{2,...,t-1\}$ such that $|U_{v_i}\cap U_{v_0}|=3$.
\end{itemize}
We will show that both $(i)$ and $(ii)$ lead to contradiction. In fact, if $(i)$ holds, call $e_r^k,e_s^\ell$ the two basis vectors in $U_{v_j}\cap U_{v_0}$. Then, the equalities $|e_r^k\cdot v_j|=|e_r^k\cdot v_0|=|e_s^\ell\cdot v_j|=|e_s^\ell\cdot v_0|=1$ are incompatible with $v_j\cdot v_0=1$. Suppose now that $(ii)$ holds and consider the vertex $v_1$. Since $v_1\cdot v_0=1$, there exists $h\in\{2,...,t-1\}$ and $k\in\{1,...,t\}$ such that $U_{v_h}\cap U_{C_k}\cap U_{v_1}\neq\emptyset$. Moreover, Since $v_1$ and $v_h$ are orthogonal to $C_k$, we have that $U_{C_k}\subseteq U_{v_1}\cap U_{v_h}$. The index $h$ must be the index $i$ in the statement of $(ii)$; otherwise we would have $U_{v_h}=U_{C_k}$, which leads $v_h\cdot v_1\neq 0$. At this point, there are two possibilities: 
\begin{itemize}
\item $|U_{v_1}\cap U_{v_i}\cap U_{v_0}|=2$. In this case there are two $(-2)$--chains, say $C_k$ and $C_\ell$, such that $U_{C_k\cup U_{C_\ell}}\subseteq U_{v_1}$. Since $v_1\cdot v_0=1$ and $v_1$ is orthogonal to $C_k$ and $C_\ell$, there must exist some $\lambda\in\Z$ such that 
$$v_1=-\lambda e_1^{k}-...-\lambda e_{c_k}^k+(\lambda +1)e_1^\ell+...+(\lambda+1)e_{c_\ell}^\ell,$$
where we use the convention fixed in Lemma~\ref{l:-2c} for the embedding of the $(-2)$--chains. Since by assumption the parameters $a_1,...,a_{t-1};c_1,...,c_t$ satisfy \textbf{(p3)}, we have that $a_1$ is even and $c_k,c_\ell$ are odd. However, $v_1\cdot v_1=-\lambda^2 c_k-(\lambda+1)^2 c_\ell$, which is in contradiction with $a_1$ being even.
\item $|U_{v_1}\cap U_{v_i}\cap U_{v_0}|=3$. This time we have three $(-2)$--chains, say $C_k,C_\ell$ and $C_p$, such that
$$v_1=\lambda e_1^{k}+...+\lambda e_{c_k}^{k}+\rho e_1^{\ell}+...+\rho e_{c_\ell}^{\ell}+\delta e_1^{p}+...+\delta e_{c_p}^{p},$$
for some $\lambda,\rho,\delta\in\Z$. Since $a_1$ is even and $c_k,c_\ell$ and $c_p$ are odd, there are two possibilities for the coefficients $\lambda,\rho$ and $\delta$: the three of them are even or two of them are odd and one is even. In both cases we have $v_1\cdot v_0=\lambda+\rho+\delta\neq 1$.
\end{itemize}
After this final contradiction we conclude that, under the assumptions of the statement, it is not possible to have $s=t-1$.

We are only left with analyzing the case $s=t+1$. Since $m=|U_w|$ and since for every $i,j\in\{2,...,t+1\}$ with $i\neq j$, we have $U_{v_i}\cap U_{v_j}=\emptyset$, we deduce that for each $(-2)$--chain $C_{k}$, $k\in\{1,...,t\}$, there must exist one and only one $i(k)\in\{2,...,s\}$ such that $U_{v_{i(k)}}\cap U_{C_{k}}\neq\emptyset$. Moreover, since there are $t$ vertices in the Wu set and the central vertex has weight $-t$, it also holds that if $i(k)=i(k')$, then $k=k'$. We fix, without loss of generality, that $i(k)=k+1$ for every $k\in\{1,...,t\}$. Since $C_{k}$ is orthogonal to $v_{k+1}$, it follows that $U_{C_{k}}\subseteq U_{v_{k+1}}$, which implies in particular that $c_{k}\leq |a_{k+1}|$ for every $k$.

The assumption $s=t+1$ and Lemma~\ref{l:-2c}\,$(1)$ imply that 
$$m=s+1+\sum_{k=1}^t (c_k-1)=2+\sum_{k=1}^t c_k=2+\sum_{k=1}^t|U_{C_k}|.$$
Therefore, since $m=|U_w|$, there are exactly two basis vectors, say $e_1^0$ and $e_2^0$, such that $e_1^0,e_2^0\in U_w$ and for every $k\in\{1,...,t\}$ we have $e_1^0,e_2^0\not\in U_{C_k}$. Since $w=\sum_{i=2}^{t}v_{i}=\sum_{j,\ell}\beta_j^\ell e_j^\ell$ with $\beta_j^\ell\in\{\pm 1\}$ and $a_2,...,a_{t+1};c_1,...,c_t$ are odd, we necessarily have that there exists one vector, say $v_2$, such that $e_1^0,e_2^0\in U_{v_2}$. Note that this implies that for every $i\in\{3,...,t+1\}$, we have $a_i=-c_{i-1}$. At this point, using Lemma~\ref{l:-2c} to fix the embedding of the $(-2)$--chains and the central vertex, it follows that the embedding of $v_0,v_3,...,v_{t+1},C_1,...,C_t$ is as claimed in Figure~\ref{f:p3b}. 

In order to conclude, we must determine the embedding of $v_1$ and $v_2$. Since $v_2\subset w$ and $\{e_1^0,e_2^0\}\cup U_{C_1}=U_{v_2}$ it follows that, up to the sign of $e_1^0$ and $e_2^0$, we have $v_2=e_1^1+...+e_{c_1}^1+e_1^0+e_2^0$. The vector $v_1$, which is the only one with even square, is connected to the central vertex, and since it is orthogonal to all the other vertices, we necessarily have $U_{C_1}\subseteq U_{v_1}$ and $U_{v_1}\cap U_{v_i}=\emptyset$ for every $i\in\{3,...,t+1\}$. Since $v_1\cdot v_2=0$ we deduce that there exist $\lambda\in\Z$ such that $v_1=e_1^1+...+e_{c_1}^1-\lambda e_1^0-(c_1-\lambda) e_2^0$.
\end{proof}

\section{Remaining cases}\label{s:step2}

The Lemmas~\ref{l:p1} and \ref{l:p2} proved in the last section imply Theorem~\ref{t:slice} for pretzel knots in families \textbf{(p1)} and \textbf{(p2)}. The ``only'' thing left now is to understand what happens with the Seifert spaces from Lemma~\ref{l:p3}. In this section we introduce the correction terms from Heegaard-Floer homology and explain how to use them in combination with Donaldson's theorem to get further obstructions to sliceness. This enhanced obstruction suffices to reduce the candidates to Seifert spaces bounding rational homology balls in Lemma~\ref{l:p3} to the double branched covers of the knots studied in Section~\ref{s:two}.

\subsection{Correction terms from Heegaard-Floer homology}\label{s:d}
Ozsv\'ath and Szab\'o, using techniques from Heegaard Floer homology, showed that if $Y$ is a rational homology sphere bounding a rational homology ball $W$, then, for each $\Spin^c$-structure $\mathfrak{s}$ on $Y$ which extends over $W$, the so called correction term $d(Y,\mathfrak{s})$ vanishes. Details can be found in \cite{b:OS}. We summarize here, following \cite{b:GJ}, a particular way of using this information as an obstruction for a Seifert space to bound a rational homology ball.

We denote by $\Spin^c(W)$ the set of $\Spin^c$ structures on the manifold $W$. This set admits a free transitive action of the group $H^2(W;\Z)$ and the following diagram, where the horizontal arrows give the action, $r$ is the restriction and $i^*$ is induced by the inclusion, commutes.
\begin{align*}
\xymatrix{
H^2(W;\Z)\times\Spin^c(W) \ar[r] \ar[d]_{(i^*,r)} 
& 
\Spin^c(W) \ar[d]^r 
\\
H^2(Y;\Z)\times\Spin^c(Y) \ar[r] 
&  
\Spin^c(Y) 
}
\end{align*}

After fixing a reference $\Spin^c$ structure on $W$, say $\mathfrak{s_0}$, we are able to identify $H^2(W;\Z)$ with $\Spin^c(W)$ as $\Spin^c(W)=\mathfrak{s_0}+H^2(W;\Z)$ and analogously $\Spin^c(Y)=r(\mathfrak{s_0})+H^2(Y;\Z)$. Therefore, it follows from the above diagram and Ozsv\'ath and Szab\'o's result that if $Y$ bounds a rational homology ball $W$ then for every $\mathfrak{s}\in r(\Spin^c(W))\subseteq\Spin^c(Y)$ the correction term $d(Y,\mathfrak{s})$ vanishes and we can identify $r(\Spin^c(W))$ with $r(\mathfrak{s_0})+i^*(H^2(W;\Z))$. We shall denote  by $V$ the subgroup $i^*(H^2(W;\Z))\subseteq H^2(Y;\Z)$.

\begin{rem}\label{r:order}
Notice that the group $V$ defined above coincides with the group $V$ in Theorem~\ref{t:CG}. There is in fact a striking formal similarity between the Casson-Gordon invariants and the $d$-invariants of a $3$-manifold. Nevertheless these two invariants detect different phenomena. For example we shall see in the next subsection that the correction terms do not obstruct the existence of a rational homology ball bounded by $Y(a,-a-2,-a-\frac{a^{2}+9}{2})$ while the Casson-Gordon invariants do obstruct (cf.\ Section~\ref{s:step3} and Remark~\ref{r:ratballs}).
\end{rem}

Given a slice knot $K\subset S^3$, the $2$-fold cover of $D^4$ branched over a slicing disc $D$ of $K$ is a rational homology ball $W_D$ \cite[Lemma~17.2]{b:Ka}. Hence, we can specialize the above discussion to $W_D$ and $Y=\partial W_D$, which is the $2$-fold cover of $S^3$ branched over $K$. In case $K$ is a Montesinos knot, let $\Gamma$ be its associated plumbing graph with $m$ vertices, and let us further assume that the associated incidence matrix $Q_\Gamma$ is negative definite. Let us call $Y=Y_\Gamma=\partial M_\Gamma$ the boundary of the plumbing $4$-manifold $M_\Gamma$ associated to $\Gamma$.

The manifold $M_\Gamma\cup_{Y_\Gamma}(-W_D)$ is a closed, oriented, negative definite, $4$-manifold and hence, by Donaldson's theorem, there exists a basis $\tilde e_1,...,\tilde e_m$ of $H^2(M_\Gamma\cup_{Y_\Gamma}(-W_D);\Z)$ with respect to which the intersection pairing is represented by $-\mathrm{Id}$. Let $\Sigma_1,...,\Sigma_m\in H_2(M_\Gamma;\Z)$ be a basis determined by the $m$ $2$-handles of $M_\Gamma$ and let $D_1,...,D_m\in H_2(M_\Gamma,Y_\Gamma;\Z)$ be a dual basis determined by their cocores. With these choices of bases, the map $\eta:H_2(M_\Gamma;\Z)\rightarrow H_2(M_\Gamma,Y_\Gamma;\Z)$ is represented by the incidence matrix $Q_\Gamma$ of the graph $\Gamma$. The application in cohomology, $\gamma:H^2(M_\Gamma,Y_\Gamma;\Z)\rightarrow H^2(M_\Gamma;\Z)$, is again represented by the matrix $Q_\Gamma$, when we consider in $H^2(M_\Gamma,Y_\Gamma;\Z)$ the basis $\Sigma_1^\ast,...,\Sigma_m^\ast$ of the Poincar\'e duals of the $\Sigma_i$ and in $H^2(M_\Gamma;\Z)$ the basis $\widetilde\Sigma_1,...,\widetilde\Sigma_m$ of the Hom-duals of the $\Sigma_i$: $\widetilde\Sigma_i(\Sigma_j)=\delta_{ij}$. The long exact sequence in cohomology of the pair $(M_\Gamma,Y_\Gamma)$ 
\begin{equation}\label{e:ex_seq_Dsec}
\xymatrix{
0\ar[r] 
&
H^2(M_\Gamma, Y_\Gamma;\Z)
\ar@{}_{\mbox{\Large$\shortparallel$}}[d] 
\ar[r]^{Q_\Gamma} 
&
H^2(M_\Gamma;\Z) 
\ar@{}_{\mbox{\Large$\shortparallel$}}[d] 
\ar[r]^{\delta} 
&
H^2(Y_\Gamma;\Z)
\ar@{}_{\mbox{\Large$\shortparallel$}}[d]
\ar[r] 
& 
0
\\
& \Z^m\ & \Z^m\ & \Z^m/Q_\Gamma\Z^m&
}
\end{equation}
allows us to identify $H^2(Y_\Gamma;\Z)$ with the cokernel of $Q_\Gamma$ via $\delta$. Therefore, we can further identify $\Spin^c(Y)=r(\mathfrak{s}_0)+\mathrm{coker}\,Q_\Gamma$. Using some algebraic topology one can obtain Theorem~3.4 in \cite{b:GJ}, which in our notation states:

\begin{thm}[Greene, Jabuka]\label{t:GJ}
Let $K,\ W_D,\ Y_\Gamma,\ M_\Gamma,\ Q_\Gamma$ be as above. Then, with the above fixed bases, $H^2(M_\Gamma\cup_{Y_\Gamma}(-W_D);\Z)\rightarrow H^2(M_\Gamma;\Z)$ has a matrix representative $A$, which leads to a factorization $Q_\Gamma=-AA^t$. Moreover, $\delta$ in \eqref{e:ex_seq_Dsec} induces an isomorphism
\[(\mathrm{im}\, A)/(\mathrm{im}\, Q_\Gamma) 
\toup^{\cong} 
i^*(H^2(W_D;\Z))\subseteq\underbrace{H^2(Y_\Gamma;\Z)}_{\mathrm{coker} Q_\Gamma}.  
\tag*{\qed}
\]
\end{thm}

From the above discussion we know that $d(Y_\Gamma,\mathfrak{s})=0$ for every $\mathfrak{s}=r(\mathfrak{s_0})+v$, where $v\in V=i^*(H^2(W_D;\Z))$. Therefore, we must have at least $|V|$  $\Spin^c$-structures on $Y_{\Gamma}$ for which the invariant $d$ vanishes. By Theorem~\ref{t:GJ} we can calculate $|V|$ as 
\begin{equation}\label{e:|V|}
\begin{split}
|V| & =  |(\mathrm{im}\, A)/(\mathrm{im}\, Q_\Gamma)|=|(A\Z^m)/(-AA^t\Z^m)|\\
& =  |\Z^m/(-A^t\Z^m)|=|\det A^t|.
\end{split}
\end{equation}

For $3$-manifolds $Y$ obtained as boundaries of $4$-dimensional plumbings specified by certain graphs, Ozsv\'ath and Szab\'o give in \cite{b:OS} an explicit formula to calculate $d(Y,\mathfrak{s})$. The class of graphs that they consider includes all star-shaped graphs $\Gamma$ such that $Q_\Gamma$ is negative definite. In order to state the formula we need to introduce some definitions and notation. With the above fixed basis, a covector $v\in H^2(M_\Gamma;\Z)$ is called \textbf{characteristic} for $Q_\Gamma$ if it is congruent modulo $2$ to the vector $(Q_{\Gamma_{11}},...,Q_{\Gamma_{mm}})$, whose coordinates in the basis $\widetilde\Sigma_1,...,\widetilde\Sigma_m$ are the elements in the diagonal of $Q_\Gamma$. We will denote by $\Char (Q_\Gamma)\subseteq H^2(M_\Gamma;\Z)$ the set of all characteristic covectors for $Q_\Gamma$. Given a cohomology class $v\in H^2(M_\Gamma;\Z)$ we will write $[v]$ to denote its equivalence class in $\mathrm{coker}\,Q_\Gamma$. For a given $\mathfrak{s}\in\Spin^c(Y_\Gamma)$ we define the set
$$\Char_{\mathfrak{s}}(Q_\Gamma):=\{v\in\Char (Q_\Gamma)|\, \mathfrak{s}=r(\mathfrak{s}_0)+[v]\}.$$
With these conventions and definitions in place, we are ready to state the formula, \cite[Corollary~1.5]{b:OS}, we use for computing correction terms:
\begin{equation}\label{e:OS}
\begin{split}
d(Y_\Gamma,\mathfrak{s})=\max_{v\in \Char_{\mathfrak{s}}(Q_\Gamma)}\frac{v^tQ_\Gamma^{-1} v+|\Gamma|}{4}.
\end{split}
\end{equation}

We already know that for every $\mathfrak{s}\in r(\mathfrak{s}_0)+V$ we have  $d(Y_\Gamma,\mathfrak{s})=0$. Nevertheless, let us compute  $d(Y_\Gamma,\mathfrak{s})$ for $\mathfrak{s}\in r(\mathfrak{s}_0)+V$ with equality \eqref{e:OS}. Greene and Jabuka's result, Theorem~\ref{t:GJ}, shows $V\cong (\mathrm{im}\,A)/(\mathrm{im}\, Q_\Gamma)$, so we consider a cohomology class $v=Ax\in H^2(M_\Gamma;\Z)$. In this case, the term $v^tQ_\Gamma^{-1} v$ from \eqref{e:OS} simplifies to
$$v^tQ_\Gamma^{-1} v=-x^tA^t(AA^t)^{-1}Ax=-x^tA^t(A^t)^{-1}A^{-1}Ax=-|x|^2,$$
and therefore
\begin{equation}\label{e:simplif}
\begin{split}
d(Y_\Gamma,\mathfrak{s})=\max_{Ax\in \Char_{\mathfrak{s}}(Q_\Gamma)}\frac{-|x|^2+m}{4}.
\end{split}
\end{equation}
where $m$ is the number of vertices in $\Gamma$. The requirement that $v=Ax$ be characteristic can be expressed as a condition on $x$ itself. In fact, by definition $v=\sum_{i=1}^m v_i\widetilde\Sigma_i$ is characteristic if and only if $v_i\equiv Q_{\Gamma ii}$ (mod 2) for all $i$. Since, by Theorem~\ref{t:GJ}, $Q_\Gamma=-AA^t$, we have that $v=Ax$ is characteristic if and only if
\begin{equation}\label{e:mod2}
\begin{split}
\sum_{j=1}^{m}A_{ij}x_j\equiv\sum_{j=1}^mA^2_{ij}\ (\mathrm{mod}\ 2)\equiv\sum_{j=1}^mA_{ij}\ (\mathrm{mod}\ 2),\s\forall i=1,...,m.
\end{split}
\end{equation}
Since $\det (Q_\Gamma)$ is the determinant of the knot $K$, it follows that $\det (A)$ is odd and therefore the matrix $A$ is invertible (mod$2$). Hence, the vector $x$(mod$2$) is uniquely determined by equivalence \eqref{e:mod2}. At the same time, it is clear that taking $x_i\equiv 1(\mathrm{mod}\ 2)\ \forall i$, the equivalences in \eqref{e:mod2} are satisfied, so it must be the unique solution. Now, since $d(Y_\Gamma,\mathfrak{s})=0$, equation \eqref{e:simplif} forces $x_i\in\{\pm 1\}$ for all $i\in\{1,...,m\}$. Notice that there are $2^m$ such vectors $v=Ax$ in $\mathrm{im}\,A$, while the number $|V|$ of equivalence classes of these vectors $[v]$ in $V\cong (\mathrm{im}\,A)/(\mathrm{im}\, Q_\Gamma)$ might be significantly smaller. 

In the next subsection we shall compute, for the Seifert spaces in Lemma~\ref{l:p3}, an upper bound $N$ on the number of $\Spin^{c}$ structures with vanishing correction terms. The above arguments imply that the inequality $|V|\leq N$ must hold, which will give further constraints on the invariants defining Seifert spaces in family \textbf{(p3)} that bound rational homology balls.

\subsection{Further study of family (p3)}

\begin{lem}\label{l:usad}
Let $a=\lambda+\rho>1$ be an odd integer, $\lambda,\rho\in\Z$ with $\lambda,\rho\geq 0$. If the Seifert space $Y=Y(-a-\lambda ^2-\rho ^2,-a-2;a)$ bounds a rational homology ball, then either $Y=Y(-a-\frac{a^2+1}{2},-a-2;a)$ or $Y=Y(-a-\frac{a^2+9}{2},-a-2;a)$. 
\end{lem}
\begin{proof}
Since $Y$ is the boundary of a rational homology ball, the intersection lattice of its associated canonical negative plumbing graph $\Psi$, which is easily seen to be negative, admits an embedding into $(\Z^m,-\mathrm{Id})$, where $m=a+2$. Let us label the vertices in $\Psi$ as follows: let $v_1,...,v_{a-1}$ be the vertices in the $(-2)$-chain corresponding to the parameter $a$, where $v_1$ is connected to the central vertex; let $v_a$ be the central vertex; let $v_{a+1}$ be the vertex with weight $-a-2$ and finally let $v_{a+2}$ be the remaining vertex. The $2$-handles represented by the vertices $\{v_1,...,v_{a+2}\}$ form a basis of $H_2(M_\Psi;\Z)$. With respect to this basis, the incidence matrix of the graph $Q_\Psi$ is the matrix of the intersection pairing on $M_\Psi$. Since $Y$ satisfies the assumptions of Lemma~\ref{l:p3} we know explicitly the factorization $Q_\Psi=-AA^t$:
$$
A^t:=
\cbra{
\begin{matrix}
1 & 0 &  &  & 0 & -1 & 1 & 1 \\ 
-1 & 1 &  &  & 0 & 0 & 1 & 1 \\ 
0 & -1 & \ddots &  & \vdots & \vdots & 1 & 1 \\ 
\vdots & \vdots &  & \ddots & \vdots & \vdots & \vdots & \vdots \\ 
0 & 0 &  &  & 1 & 0 & 1  & 1 \\ 
0 & 0 &  &  & -1 & 0 & 1 & 1 \\ 
0 & 0 & \hdots  & \hdots & 0 & 0 & 1 & -\lambda \\ 
0 & 0 & \hdots & \hdots & 0 & 0 & 1 & -\rho
\end{matrix} 
}.
$$

By assumption, there exists a rational homology ball $W$ such that $Y=\partial W$. With the notation and conventions fixed in Section~\ref{s:d}, we have that every $\mathfrak{s}\in\Spin^c(Y)$ which extends to $W$ is of the form $r(\mathfrak{s}_0)+v$, where $v\in V\subset H^2(Y;\Z)$. Greene and Jabuka's result, Theorem~\ref{t:GJ}, states that $V\cong (\mathrm{im}\, A)/(\mathrm{im}\, Q_\Psi)$ and by \eqref{e:|V|} we have
\begin{align*}
|V| &=|\det A|\\
& = \sqrt{|\det Q_\Psi|}\\
 &= \sqrt{\abs{\cbra{\frac{1}{a}+\frac{1}{-a-2}+\frac{1}{-a-\lambda ^2-\rho ^2}}a(-a-2)(-a-\lambda ^2-\rho ^2)}}\\
 &=|\lambda-\rho|.
\end{align*}
Notice that we have used equality \eqref{e:det} to calculate $|\det Q_\Psi|$. 

From the discussion in Section~\ref{s:d} we know that, for every cohomology class $v=Ax\in\mathrm{im}\,A$ for which $\mathfrak{s}=r(\mathfrak{s}_0)+v$ satisfies $d(Y,\mathfrak{s})=0$, it holds that $x=(x_1,...,x_m)$ has coordinates $x_i\in\{\pm 1\}$ for all $i$. While there are $2^m$ such $v=Ax$ in $\mathrm{im}\,A$, there are significantly fewer equivalence classes of such vectors $[v]$ in $V=(\mathrm{im}\, A)/(\mathrm{im}\, Q_\Psi)$. In fact, observe that any two cohomology classes, $v=Ax$ and $v'=Ax'$ satisfy $[v]=[v']$ if and only if there exists some $y$ such that $A(x-x')=-AA^ty$. Namely, if and only if $x-x'\in\mathrm{im}\,A^t$. The first $a+1$ columns $\mathrm{col}_1,...,\mathrm{col}_{a+1}$ of $A^t$ generate the kernel of $\ell:\Z^m\longrightarrow\Z$, where $\ell(x_1,...,x_{a+2})=x_{a+1}-x_{a+2}$. In fact, since $\ell(\mathrm{col}_i)=0$ for $i=1,...,a+1$ we have $\mathrm{span}(\mathrm{col}_1,...,\mathrm{col}_{a+1})\subseteq\mathrm{Ker}\,\ell$. On the other hand, for every $y=(y_1,...,y_{a+2})\in\mathrm{Ker}\,\ell$ it holds $y_{a+1}=y_{a+2}$ and since the determinant of the leading principal minor of $A^t$ is $-1$, it follows that $\mathrm{span}(\mathrm{col}_1,...,\mathrm{col}_{a+1})=\mathrm{Ker}\,\ell$. Therefore, $\mathrm{Ker}\,\ell\subset\mathrm{im}\, A^t$ and hence, if $\ell(x)=\ell(x')$ then $[v]=[v']$. The functional $\ell$, when restricted to the set $\{x\in\Z^{a+2}|\, x_i\in\{\pm 1\}\}$, only takes $3$ different values. Thus, $d(Y,\mathfrak{s})$ vanishes for at most $3$ $\Spin^c$ structures in $r(\mathfrak{s}_0)+V$. On the other hand, we know that all the $\Spin^c$ structures of the form $\mathfrak{s}=r(\mathfrak{s}_0)+v$, with $v\in V$, extend to $W$ and therefore, they all have vanishing correction terms. We conclude that $|V|=|\lambda-\rho|\leq 3$. Since $a$ is odd we cannot have $\lambda=\rho +2$, and therefore $\lambda=\rho +1$ or $\lambda=\rho +3$. Taking into account that $a=\lambda+\rho$ and making the pertinent substitutions we obtain the Seifert spaces in the statement.
\end{proof}

\begin{cor}\label{c:duda}
Every Seifert space in Lemma~\ref{l:p3} bounding a rational homology ball is either of the form $Y=Y(-c_1-\frac{c_1^2+1}{2},-c_{1}-2,-c_2,...,-c_{t};c_1,...,c_t)$ or $Y=Y(-c_1-\frac{c_1^2+9}{2},-c_{1}-2,-c_2,...,-c_{t};c_1,...,c_t)$.
\end{cor}
\begin{proof}
The Seifert spaces in Lemma~\ref{l:p3} are of the form
$$Y=Y(-c_{1}-\lambda^2-\rho^2,-c_{1}-2,-c_2,...,-c_{t};c_1,...,c_t),$$ 
for some $\lambda,\rho\in\Z$ such that $\lambda+\rho=c_1$. Since $Y$ is invariant under any change of order of the parameters, it is the double cover of $S^3$ branched over the pretzel knot 
$$K:=P(-c_2,c_2,-c_3,c_3,...,-c_{t},c_{t},-c_{1}-\lambda^2-\rho^2,-c_{1}-2,c_1).$$ 
Adding to the projection of $K$ the band shown in Figure~\ref{f:cob_pretzel} and performing a ribbon move along it, we obtain the disjoint union of an unknot and the pretzel knot 
$$K_2:=P(-c_3,c_3,...,-c_{t},c_{t},-c_{1}-\lambda^2-\rho^2,-c_{1}-2,c_1).$$
Therefore, $K\sim K_2$, where $\sim$ stands for concordance equivalence. Since concordance is an equivalence relation, we have $K\# -\overline K_2\sim K_2\# -\overline K_2$ and hence $K\#-\overline K_2$ is slice. It follows that its double branched cover, the $3$--manifold $Y\# -Y_2$, is the boundary of a rational homology ball, and therefore there exists a rational homology cobordism between $Y$ and 
$$Y_2:=Y(-c_3,c_3,...,-c_{t},c_{t},-c_{1}-\lambda^2-\rho^2,-c_{1}-2,c_1).$$
Obviously, we can repeat the same construction $t-1$ times obtaining that there exists a rational homology cobordism between $Y$ and the Seifert space with three exceptional fibers $Y(-c_{1}-\lambda^2-\rho^2,-c_{1}-2,c_1)$. Among these spaces, by Lemma~\ref{l:usad}, the only ones that possibly bound a rational homology ball are of the form $Y=Y(-c_1-\frac{c_1^2+1}{2},-c_1-2;c_1)$ and $Y=Y(-c_1-\frac{c_1^2+9}{2},-c_1-2;c_1)$.
\end{proof}

\begin{rem}\label{r:concordance}
The above corollary shows that the Seifert space \[Y(a_{1},\dots,a_{k},\alpha_{1},\dots,\alpha_{r},-\alpha_{1},\dots,-\alpha_{r}),\] where $a_{i},\alpha_{j}\in\Z$, bounds a rational homology ball if and only if $Y(a_{1},\dots,a_{k})$ does.
\end{rem}

\subsection{Proof of main Theorem}
In this last section we proof Theorem~\ref{t:slice} putting together the conclusions of the preceding sections. We shall give two nearly equivalent statements of our main result: one regarding the sliceness of pretzel knots, as stated in the Introduction, and another one focusing on the Seifert spaces bounding rational homology balls. Notice that in this second case our results are slightly weaker since we were able to determine, in Section~\ref{s:two}, that many of the knots in family $P(a,-a-2,-\frac{(a+1)^{2}}{2})$, $a\geq 3$ odd, are not slice but this does not imply that the double branched covers do not bound rational homology balls.

\begin{proof}[Proof of Theorem~\ref{t:slice}]
Up to considering mirror images we can suppose that the double branched cover of the slice knot $K$ bounds a negative definite four manifold. Since $K$ is slice $\sigma (K)=0$ and its double branched cover $Y$ bounds a rational homology ball. It follows that the slice conditions (SC) from Section~\ref{s:1step} are satisfied.

If $K$ is of type \textbf{(p1)} or \textbf{(p2)} then the theorem is an immediate corollary from Lemmas~\ref{l:p1} and \ref{l:p2}, since the Seifert spaces determine the pretzel knots up to reordering of the parameters. 

If $K$ belongs to \textbf{(p3)} then its double branched cover $Y$ is one of the Seifert spaces in Corollary~\ref{c:duda}. If $Y=Y(-c_1-\frac{c_1^2+9}{2},-c_{1}-2,-c_2,...,-c_{t};c_1,...,c_t)$ then by  Remarks~\ref{r:concordance} and \ref{r:ratballs}  it does not bound a rational homology ball and therefore $K$ is not slice. Finally, if $Y=Y(-c_1-\frac{c_1^2+1}{2},-c_{1}-2,-c_2,...,-c_{t};c_1,...,c_t)$ then it is the double branched cover of all pretzel knots with defining parameters $\{-c_1-\frac{c_1^2+1}{2},-c_{1}-2,-c_2,...,-c_{t};c_1,...,c_t\}$. Since the Alexander polynomial is invariant under mutation, all the different pretzel knots with the same set of parameters have the same Alexander polynomial. By assumption $K$ is slice and therefore its Alexander polynomial has a Fox-Milnor factorisation. It then follows that the Alexander polynomial of any knot concordant to $K$ also has a Fox-Milnor factorisation. Since  the pretzel knot $P(c_{1},-c_{1}-2,-\frac{(c_1+1)^{2}}{2},c_{2},-c_{2},\dots,c_{t},-c_{t})$ is clearly concordant to $P(c_{1},-c_{1}-2,-\frac{(c_1+1)^{2}}{2})$ and  by assumption $c_{1}\not\in\mathcal{E}$ by Theorem~\ref{t:Alex} we conclude that $K$ is not slice.
\end{proof}

As far as Seifert spaces bounding rational homology balls are concerned we have the following result which is just a rephrasing of Theorem~\ref{t:slice}.

\begin{thm} 
Consider a Seifert space of the form $Y=Y(p_{1},\dots,p_{n})$ with $n\geq 3$, $p_{1}$ even and $p_{i}$ odd for $i=2,\dots,n$. If $Y$ is the boundary of a rational homology ball, then the set of parameters $\mathcal P=\{p_{1},\dots,p_{n}\}$ verifies one of the following:
\begin{enumerate}
\item $n$ is even and $\mathcal P$ can be written as $\{p_{1},-p_{1}\pm 1,q_{1},-q_{1},\dots,q_{\frac{n}{2}},-q_{\frac{n}{2}}\}$
\item $n$ is odd and $\mathcal P$ can be written as $\{p_{1},q_{1},-q_{1},\dots,q_{\frac{n-1}{2}},-q_{\frac{n-1}{2}}\}$
\item $n$ is odd and $\mathcal P$ can be written as 
 \begin{enumerate}
  \item $\{a,-a-2,-\frac{(a+1)^{2}}{2},q_{1},-q_{1},\dots,q_{\frac{n-3}{2}},-q_{\frac{n-3}{2}}\}$ or 
  \item $\{-a,a+2,\frac{(a+1)^{2}}{2},q_{1},-q_{1},\dots,q_{\frac{n-3}{2}},-q_{\frac{n-3}{2}}\}$ 
 \end{enumerate} 
  where $a\geq 3$ and $|p_{1}|=\frac{(a+1)^{2}}{2}$.
\end{enumerate} 
Moreover, in the two first cases the rational homology ball exists.
\end{thm}

\appendix
\section{}
In this appendix we compute the Alexander polynomial of the  pretzel knots $P_{a}=P(a,-a-2,-\frac{(a+1)^{2}}{2})$ where $a>1$ odd. To this end we shall heavily rely on Jabuka's theorem \cite[Theorem~3.1]{b:Ja} which gives in particular a symmetrized linking form $\mathcal L + \mathcal L^\tau$ of the pretzel knots $P_{a}$ associated to the Seifert surfaces given in \cite[Figure~2]{b:Ja}. We shall use the following notation, where $A_{n}$ is a square matrix of size $n$:
$$
A_n = \left(
\begin{array}{cccccc}
t-1 & -1 & -1 & ... & -1 & -1 \\
t & t-1 &  -1 & ... & -1 & -1 \\
t & t & t-1 &  ... & -1 & -1 \cr
\vdots & \vdots & \vdots   & \ddots & \vdots & \vdots  \cr 
t & t & t &  ... & t-1 & -1  \cr
t & t & t &  ... & t & t-1 \cr
\end{array}
\right).
$$

\begin{thm}[Jabuka]
The Alexander polynomial $\Delta_{a}(t)$ of the knots in family $P_{a}$ satisfies $\Delta_{a}(t)\doteq\det\Theta_{a}$ where %
$$
\Theta_{a}=\mathcal L -t \mathcal L^\tau  = \left(
\begin{array}{ccc|ccc|c|c}
   &       &   &   &           &      &  -1        & 0                   \cr 
   & A_{a+1} &   &   & 0         &      &  \vdots    & \vdots              \cr 
   &       &   &   &           &      &   -1       & 0                   \cr \hline 
   &       &   &   &           &      &    1       & 0                   \cr
   & 0     &   &   & -A_{a-1}  &      &   \vdots   & \vdots              \cr
   &       &   &   &           &      &   1        & 0                   \cr \hline 
 t & \dots & t &-t & \dots     & -t   & 0          &   -t                        \cr \hline 
 0 & \dots & 0 & 0 & \dots     & 0    & 1          &  \tfrac{(a+1)^{2}}{4}(1-t)  \cr 
\end{array} 
\right).
$$
\end{thm}
Expanding the determinant of $\Theta_{a}$ by the last column we obtain
\begin{equation}\label{e:Alexp}
\Delta_{a}(t)\doteq\det\Theta_{a}=t\det A_{a+1}\det A_{a-1}+\frac{(a+1)^{2}}{4}(1-t)\det B_{a},
\end{equation} 
where $B_{a}$ equals $\Theta_{a}$ with the last row and column removed. In the following two lemmas we shall compute the determinants of $A_{n}$ and $B_{n}$ for all $n\in\N$. 

\begin{lem}\label{l:detA}
We have $\det A_{n}=\sum_{i=0}^{n}(-1)^{n-i}t^{i}$
\end{lem}
\begin{proof}
The statement is trivially true for $A_{1}$. Let us suppose the lemma holds for $A_{1},\dots,A_{n-1}$ and we will show that it is then true for $A_{n}$ finishing the proof. After subtracting the second row from the first in $A_{n}$ expand the determinant by the first row to obtain
$$\det A_{n}=-\det A_{n-1}+t\det C_{n-1},$$
where the entries of the matrix $C_{n}$ are defined as
$$(c_{n})_{ij}=
\left\{ 
\begin{array}{l l}
(c_{n})_{11}=t\\
(c_{n})_{ij}=(a_{n})_{ij}\ \mbox{for } (i,j)\neq (1,1).
 \end{array} \right.
$$
Remark that the determinant of $C_{n}$ satisfies $\det C_{n}=t\det C_{n-1}$ and therefore $\det C_{n}=t^{n}$. We have then $\det A_{n}=-\det A_{n-1}+t^{n}$ concluding the proof.
\end{proof}

\begin{lem}
It holds $\det B_{n}=t^{n+1}-t^{n}$
\end{lem}
\begin{proof}
A series of evident row and column operations on $B_{n}$ yield the following matrix
{\footnotesize
$$
B_{n}' = \left(
\begin{array}{ccc|ccc|c}
   &       &   &   &           &      &  0          \cr 
   & A_{n+1}' &   &   & 0         &    &  \vdots       \cr 
      &       &   &   &           &      &  0          \cr 
   &       &   &   &           &      &   -1          \cr \hline 
   &       &   &   &           &      &    0          \cr
   & 0     &   &   & -A_{n-1}'  &      &   \vdots      \cr
      &       &   &   &           &      &    0          \cr
   &       &   &   &           &      &   1           \cr \hline 
 0 & \dots  & 0\ t & 0 & \dots   &0 -t   &   0           \cr 
\end{array} 
\right),
$$
}
where $\det A'_{n}=\det A_{n}$. Expanding the determinant of $B_{n}'$ by the last column we obtain
\begin{align*}
\det B_{n}=\,&-\big(-t\det A_{n+1}\det (-A_{n-2})\pm t\det C_{1}\big)\\
&+(-1)^{n+1}\left((-1)^{n-1}t\det A_{n}\det(-A_{n-1})- t\det C_{2}\right),
\end{align*}
where $C_{1}$ and $C_{2}$ are easily seen to have vanishing determinant. By the proof of Lemma~\ref{l:detA} it then follows
\begin{align*}
\det B_{n} & =(-1)^{n}t\left(\det A_{n+1}\det A_{n-2}-\det A_{n}\det A_{n-1}\right)\\
&=(-1)^{n}t\left((t^{n+1}-\det A_{n})\det A_{n-2}-\det A_{n}(t^{n-1}-\det A_{n-2}) \right)\\
&=(-1)^{n}t^{n}\left(t^{2}\det A_{n-2}-\det A_{n}\right)\\
&=(-1)^{n}t^{n}\left((-1)^{n}(t-1)\right)\\
&=t^{n+1}-t^{n}.
\end{align*}
\end{proof}

Notice that for $n$ even we have $\det A_{n}=\frac{t^{n+1}+1}{t+1}$. Substituting in the expression~\ref{e:Alexp} the results of the last two lemmas we obtain, since $a$ is odd, 
\begin{align*}
 \Delta_{a}(t)&\doteq \frac{t^{a+2}+1}{t+1}\frac{t^a+1}{t+1}-\frac{(a+1)^{2}}{4}t^{a-1}(t-1)^{2}\\
&=\prod_{
\substack{d|a+2\\ d\neq 1}
}\Phi_{d}(-t)
\prod_{
\substack{\delta|a\\\delta\neq 1}
}\Phi_{\delta}(-t)-\frac{(a+1)^{2}}{4}t^{a-1}(t-1)^{2},
\end{align*}
where $\Phi_{n}$ stands for the $n$-th cyclotomic polynomial.

\bibliographystyle{amsalpha}
\bibliography{bibliografia}

\vspace{1cm}
\end{document}

%% file: surgery.eps_tex
\begingroup%
  \makeatletter%
  \providecommand\color[2][]{%
    \errmessage{(Inkscape) Color is used for the text in Inkscape, but the package 'color.sty' is not loaded}%
    \renewcommand\color[2][]{}%
  }%
  \providecommand\transparent[1]{%
    \errmessage{(Inkscape) Transparency is used (non-zero) for the text in Inkscape, but the package 'transparent.sty' is not loaded}%
    \renewcommand\transparent[1]{}%
  }%
  \providecommand\rotatebox[2]{#2}%
  \ifx\svgwidth\undefined%
    \setlength{\unitlength}{341.85bp}%
    \ifx\svgscale\undefined%
      \relax%
    \else%
      \setlength{\unitlength}{\unitlength * \real{\svgscale}}%
    \fi%
  \else%
    \setlength{\unitlength}{\svgwidth}%
  \fi%
  \global\let\svgwidth\undefined%
  \global\let\svgscale\undefined%
  \makeatother%
  \begin{picture}(1,0.17383766)%
    \put(0,0){\includegraphics[width=\unitlength]{surgery.eps}}%
    \put(0.63388298,0.15624482){\color[rgb]{0,0,0}\makebox(0,0)[lb]{\smash{{\scriptsize $-\tfrac{a^2+9}{2}$}}}}%
    \put(0.21762789,0.13775098){\color[rgb]{0,0,0}\makebox(0,0)[lb]{\smash{{\scriptsize $-1$}}}}%
    \put(0.28373333,0.04450711){\color[rgb]{0,0,0}\makebox(0,0)[lb]{\smash{{\scriptsize $-2$}}}}%
    \put(0.37369363,0.04188295){\color[rgb]{0,0,0}\makebox(0,0)[lb]{\smash{{\scriptsize $-2$}}}}%
    \put(0.76142708,0.07853782){\color[rgb]{0,0,0}\makebox(0,0)[lb]{\smash{{\scriptsize $a$}}}}%
    \put(0.81177126,0.15780815){\color[rgb]{0,0,0}\makebox(0,0)[lb]{\smash{{\scriptsize $-2$}}}}%
    \put(0.52149095,0.05763057){\color[rgb]{0,0,0}\makebox(0,0)[lb]{\smash{   
 {\small blow downs}}}}%
    \put(0.35982664,0.15586488){\color[rgb]{0,0,0}\makebox(0,0)[lb]{\smash{{\scriptsize $a-1$}}}}%
    \put(0.05636867,0.03761811){\color[rgb]{0,0,0}\makebox(0,0)[lb]{\smash{{\scriptsize $-a-2$}}}}%
    \put(0.47078197,0.027347){\color[rgb]{0,0,0}\makebox(0,0)[lb]{\smash{{\scriptsize $-2$}}}}%
    \put(0.0021519,0.15647825){\color[rgb]{0,0,0}\makebox(0,0)[lb]{\smash{{\scriptsize $-a-\tfrac{a^2+9}{2}$}}}}%
    \put(0.89011337,0.08609586){\color[rgb]{0,0,0}\makebox(0,0)[lb]{\smash{{\small $=: L_a$}}}}%
    \put(0.6523851,0.12717841){\color[rgb]{0,0,0}\makebox(0,0)[lb]{\smash{{\scriptsize $\mu_1$}}}}%
    \put(0.85832338,0.13185882){\color[rgb]{0,0,0}\makebox(0,0)[lb]{\smash{{\scriptsize $\mu_2$}}}}%
  \end{picture}%
\endgroup%

%% file: Ccomplex.eps_tex
\begingroup%
  \makeatletter%
  \providecommand\color[2][]{%
    \errmessage{(Inkscape) Color is used for the text in Inkscape, but the package 'color.sty' is not loaded}%
    \renewcommand\color[2][]{}%
  }%
  \providecommand\transparent[1]{%
    \errmessage{(Inkscape) Transparency is used (non-zero) for the text in Inkscape, but the package 'transparent.sty' is not loaded}%
    \renewcommand\transparent[1]{}%
  }%
  \providecommand\rotatebox[2]{#2}%
  \ifx\svgwidth\undefined%
    \setlength{\unitlength}{51.925bp}%
    \ifx\svgscale\undefined%
      \relax%
    \else%
      \setlength{\unitlength}{\unitlength * \real{\svgscale}}%
    \fi%
  \else%
    \setlength{\unitlength}{\svgwidth}%
  \fi%
  \global\let\svgwidth\undefined%
  \global\let\svgscale\undefined%
  \makeatother%
  \begin{picture}(1,1.12421762)%
    \put(0,0){\includegraphics[width=\unitlength]{Ccomplex.eps}}%
  \end{picture}%
\endgroup%

%% file: alternativa.eps_tex
\begingroup%
  \makeatletter%
  \providecommand\color[2][]{%
    \errmessage{(Inkscape) Color is used for the text in Inkscape, but the package 'color.sty' is not loaded}%
    \renewcommand\color[2][]{}%
  }%
  \providecommand\transparent[1]{%
    \errmessage{(Inkscape) Transparency is used (non-zero) for the text in Inkscape, but the package 'transparent.sty' is not loaded}%
    \renewcommand\transparent[1]{}%
  }%
  \providecommand\rotatebox[2]{#2}%
  \ifx\svgwidth\undefined%
    \setlength{\unitlength}{380.825bp}%
    \ifx\svgscale\undefined%
      \relax%
    \else%
      \setlength{\unitlength}{\unitlength * \real{\svgscale}}%
    \fi%
  \else%
    \setlength{\unitlength}{\svgwidth}%
  \fi%
  \global\let\svgwidth\undefined%
  \global\let\svgscale\undefined%
  \makeatother%
  \begin{picture}(1,0.16076938)%
    \put(0,0){\includegraphics[width=\unitlength]{alternativa.eps}}%
    \put(0.00339728,0.1436179){\color[rgb]{0,0,0}\makebox(0,0)[lb]{\smash{{\scriptsize $-\tfrac{a^2+9}{2}$}}}}%
    \put(0.16345854,0.14835249){\color[rgb]{0,0,0}\makebox(0,0)[lb]{\smash{{\scriptsize $-2$}}}}%
    \put(0.79934208,0.0717475){\color[rgb]{0,0,0}\makebox(0,0)[lb]{\smash{{\scriptsize $-a\!\!-\!\!2$}}}}%
    \put(0.22922284,0.14466825){\color[rgb]{0,0,0}\makebox(0,0)[lb]{\smash{{\scriptsize $-\tfrac{(a+2)^2+9}{2}$}}}}%
    \put(0.59796602,0.11365622){\color[rgb]{0,0,0}\makebox(0,0)[lb]{\smash{{\scriptsize $-2$}}}}%
    \put(0.39495116,0.1086684){\color[rgb]{0,0,0}\makebox(0,0)[lb]{\smash{{\scriptsize $-2$}}}}%
    \put(0.87703582,0.14187922){\color[rgb]{0,0,0}\makebox(0,0)[lb]{\smash{{\scriptsize $-2$}}}}%
    \put(0.67717182,0.14474104){\color[rgb]{0,0,0}\makebox(0,0)[lb]{\smash{{\scriptsize $-\tfrac{(a+2)^2+9}{2}$}}}}%
    \put(0.49307431,0.14506867){\color[rgb]{0,0,0}\makebox(0,0)[lb]{\smash{{\scriptsize $-\tfrac{(a+2)^2+9}{2}$}}}}%
    \put(0.90785585,0.0711444){\color[rgb]{0,0,0}\makebox(0,0)[lb]{\smash{  }}}%
    \put(0.55439318,0.0717475){\color[rgb]{0,0,0}\makebox(0,0)[lb]{\smash{{\scriptsize $-a$}}}}%
    \put(0.3578298,0.07219161){\color[rgb]{0,0,0}\makebox(0,0)[lb]{\smash{{\scriptsize $a$}}}}%
    \put(0.10409951,0.07210417){\color[rgb]{0,0,0}\makebox(0,0)[lb]{\smash{{\scriptsize $a$}}}}%
    \put(0.20764131,0.10427428){\color[rgb]{0,0,0}\makebox(0,0)[lb]{\smash{{\small handle} }}}%
    \put(0.45867524,0.08956937){\color[rgb]{0,0,0}\makebox(0,0)[lb]{\smash{{\small isotopy}}}}%
    \put(0.66034268,0.08956937){\color[rgb]{0,0,0}\makebox(0,0)[lb]{\smash{{\small isotopy}}}}%
    \put(0.20790389,0.08379243){\color[rgb]{0,0,0}\makebox(0,0)[lb]{\smash{{\small slide}}}}%
    \put(0.91952915,0.0756007){\color[rgb]{0,0,0}\makebox(0,0)[lb]{\smash{$:=\tilde L_a$}}}%
  \end{picture}%
\endgroup%